\documentclass[a4paper]{article}
\usepackage[english]{babel}

\usepackage[pdftex]{graphicx}
\usepackage{graphicx}
\usepackage{epstopdf} 
\usepackage[normal]{subfigure}
\usepackage{color}
\usepackage{amsfonts,amsmath,amssymb,amsthm,mathtools}
\usepackage{float}
\usepackage[inline]{enumitem}
\usepackage{a4wide}
\usepackage[pagewise,mathlines]{lineno}%\linenumbers

\newtheorem{theorem}{Theorem}

\newtheorem{lemma}[theorem]{Lemma}
\theoremstyle{remark}\newtheorem{remark}[theorem]{Remark}

\def\T{\textrm{T}}

\newcommand{\RR}{\mathbb{R}}

\newcommand{\ie}{{\it i.e.}}

\begin{document}
\title{Structure preserving schemes for Fokker-Planck equations with nonconstant diffusion matrices}

\author{Nadia Loy \thanks{Department of Mathematical Sciences ``G. L. Lagrange'', Politecnico di Torino, Corso Duca degli Abruzzi 24, 10129 Torino, Italy,                 (\texttt{nadia.loy@polito.it})}\and
        Mattia Zanella\thanks{Department of Mathematics "F. Casorati", Via A. Ferrata 5, 27100 Pavia, Italy
                (\texttt{mattia.zanella@unipv.it})}}
                		
\date{}

\maketitle

\begin{abstract}
In this work we consider an extension of a recently proposed structure preserving numerical scheme for nonlinear Fokker-Planck-type equations to the case of nonconstant full diffusion matrices. While in existing works the schemes are formulated in a one-dimensional setting, here we consider exclusively the two-dimensional case. We prove that the proposed schemes preserve fundamental structural properties like nonnegativity of the solution without restriction on the size of the mesh and entropy dissipation. Moreover, all the methods presented here are at least second order accurate in the transient regimes and arbitrarily high order for large times in the hypothesis in which the flux vanishes at the stationary state. Suitable numerical tests will confirm the theoretical results. \\

\end{abstract}

\tableofcontents

\section{Introduction}
We are interested in nonlinear Fokker-Planck equations describing the evolution of a multivariate distribution function $f(w,t)\ge 0$, with $t\geq 0, w \in \Omega \subseteq \mathbb{R}^{2}$ of the following form
\begin{equation}\label{FP1}
 \begin{cases}\vspace{0.5em}
 \partial_t f(w,t)=\nabla_{w}\cdot \mathcal{F}\textcolor{black}{(w,t)}, \quad \mathcal{F}\textcolor{black}{(w,t)}= \mathcal{B}[f](w,t)f(w,t)+\nabla_w \cdot \big( \mathbb{D}(w)f(w,t)\big), \, w\in\Omega,
 \\
\mathcal{F}(w,t)\cdot \boldsymbol{n}(w)=0, \quad w \in \partial \Omega,\\
 f(w,0)=f_0(w), \, w \in \Omega,
 \end{cases}
\end{equation}
where $\Omega \subset \mathbb{R}^{\textcolor{black}{2}}$ is bounded and $\textbf{n}(w)$ is the outward normal unit vector defined for $w \in \partial \Omega$. In particular the no-flux boundary condition $\mathcal{F}\textcolor{black}{(w,t)}\cdot \boldsymbol{n}(w)=0, \, w \in \partial \Omega$ guarantees conservation of mass in $\Omega$, i.e. $\int_{\Omega}f(w,t) \, dw=\int_{\Omega}f_0(w) \, dw \, \forall t\ge0$.
The drift term $\mathcal{B}[f](\cdot,t)$ can classically be defined as the following nonlocal bounded operator \begin{equation}\label{B_operator}
\begin{array}{rl}
\mathcal{B}[f](\cdot,t): &\Omega \longmapsto \mathbb{R}^{\textcolor{black}{2}} \\[8pt]
 &w \longmapsto \mathcal{B}[f](w,t)=\displaystyle \int_{\Omega}P(w,w_*)(w-w_*)f(w_*,t) dw_*,
\end{array}
\end{equation}
where $P(\cdot,\cdot):\Omega\times \Omega\rightarrow \RR^+$. Therefore, the drift term $\mathcal{B}[f]\textcolor{black}{(w,t)}$ depends on time only through $f\textcolor{black}{(w,t)}$.
In \eqref{FP1} we consider a nonconstant diffusion matrix $\mathbb{D}(w)$ which is supposed to be symmetric and positive definite in the interior of $\Omega$, while it vanishes at the border, i.e.
\begin{equation}\label{eq:Diff_nulla}
\mathbb{D}(w)=0, \quad w \in \partial \Omega.
\end{equation}

%%%%%%%%%%%%%%%%%%%%%%%%%%%%

Kinetic-type equations with general diffusion often arise in the derivation of aggregate descriptions of many particles systems. Without intending to review the very huge literature on this topic we mention \cite{BCC,BCH,BDZ,CFTV,DFT} for applications to collective phenomena, \cite{CHP,LP,OL,YEE} for related models in self-organized biological aggregations, and \cite{Furioli,Tosc1,PT,LT,TV,TZ} for their relation with Boltzmann-type modelling. Kinetic equations have a strong physical interpretation as they describe the time evolution of probability density functions that describe the statistical distribution of the microscopic variables. Therefore, their solution should be nonnegative. Moreover, their trend to equilibrium is studied through an entropy functional that is dissipated in time and minimized by a unique stationary equilibrium. The necessity to deal with a general diffusion matrix arises from various applications where heterogeneity appears in the evolution of the distribution function. \textcolor{black}{Of course, this gives rise to a genuinely multi-dimensional problem whose stationary state makes the divergence of the flux vanish.}

\textcolor{black}{In this manuscript we concentrate on the \textcolor{black}{the two-dimensional problem and on the }special case where also the flux of the problem vanishes at the stationary equilibrium.} We will develop in this setting finite difference numerical schemes for the introduced problem that preserve structural properties like nonnegativity of the solution, entropy dissipation and that approximate with arbitrary accuracy the stationary state of the problem. 
Furthermore, the methods here developed are second order accurate in the transient regime and do not require restrictions on the mesh size. The schemes \textcolor{black}{with the introduced features are usually referred to as structure preserving schemes (SP)}. The methods here derived are based on recent works \textcolor{black}{in} this direction \cite{CCP,DPZ,Pareschi_Zanella,PZ2} and follow pioneering works on linear Fokker-Planck equations \cite{CC,LLPS}, see also \cite{BCDS,BD,FPR,SG,SSLMM}. We refer to \cite{BCF,BCHR,CCH,CHJS,Q} for related methods in the case of degenerate diffusion and to \cite{Gosse} for a recent survey on methods preserving steady states of balance laws.

In more details the paper is organized as follows. In Section \ref{sect:SP} we derive the structure preserving  scheme. We will compare the obtained scheme with recent results for \textcolor{black}{1D problems}. Hence, in Section \ref{sect:prop} we prove nonnegativity of the numerical solution in the case of explicit and semi-implicit time integration. Sufficient conditions will be explicated in terms of bounds on the time step. 
The trend to equilibrium is then investigated in Section \ref{sect:entropy} \textcolor{black}{in the case of linear problems}, here we prove that the constructed SP scheme dissipates the numerical entropy. Finally in Section \ref{sect:applications} we present several applications of the schemes in Fokker-Planck problems describing emerging patterns in interacting systems.
 Some conclusions are reported at the end of the manuscript. 

\section{Structure preserving schemes and \textcolor{black}{full nonconstant diffusion matrices}}\label{sect:SP}

In this section we focus on the design of a numerical scheme for the nonlinear Fokker-Planck equation with general diffusion matrix of the form \eqref{FP1}.   \\
\subsection{Stationary states}
We first observe that in \eqref{FP1} the two dimensional flux $\mathcal{F}\textcolor{black}{(w,t)}=\left[ \mathcal{F}^x(w,t),\mathcal{F}^y(w,t)\right]^\T$ is given by 
\begin{equation}\label{flux.an}
\mathcal{F}(w,t)=\mathcal{C}(w,t)f(w,t)+\mathbb{D}(w)\nabla_w f(w,t), 
\end{equation}
with $\mathbb D(w)$ a nonconstant diffusion matrix of the form 
$$
\mathbb{D}(w)=\left[\begin{matrix}
\mathbb{D}^{1,1}(w) \quad \mathbb{D}^{1,2}(w) \\[3pt]
\mathbb{D}^{2,1}(w) \quad \mathbb{D}^{2,2}(w)
\end{matrix}\right],
$$
\textcolor{black}{such that $\mathbb D  \in \mathcal{C}^2(\Omega)$ and is symmetric and positive definite in the interior of $\Omega$. Since we considered $\mathbb{D}\textcolor{black}{(w)}$ symmetric and positive definite its determinant is strictly positive, i.e. $|\mathbb{D}\textcolor{black}{(w)}|>0,  \forall \textcolor{black}{w \, \in \Omega}$. We remark that,
as $|\mathbb{D}\textcolor{black}{(w)}|=\mathbb{D}_{1,1}\textcolor{black}{(w)}\mathbb{D}_{2,2}\textcolor{black}{(w)}-\mathbb{D}_{1,2}\textcolor{black}{(w)}\mathbb{D}_{2,1}\textcolor{black}{(w)}=\mathbb{D}_{1,1}\textcolor{black}{(w)}\mathbb{D}_{2,2}\textcolor{black}{(w)}-\mathbb{D}_{1,2}^2\textcolor{black}{(w)}$, the diagonal elements of the diffusion matrix, i.e. $\mathbb{D}^{1,1}(w)$ and $\mathbb{D}^{2,2}(w)$, cannot vanish $\forall w\in\Omega$.} 
Moreover, in \eqref{flux.an} we considered $\mathcal{C}(w,t)=\left[\mathcal{C}^x(w,t),\mathcal{C}^y(w,t)\right]^\T$ where
$$
\mathcal{C}^x(w,t)=\mathcal{B}[f]^x(w,t)+\partial_{x}\mathbb{D}^{1,1}(w)+\partial_{y}\mathbb{D}^{2,1}(w),
$$
$$
\mathcal{C}^y(w,t)=\mathcal{B}[f]^y(w,t)+\partial_{x}\mathbb{D}^{1,2}(w)+\partial_{y}\mathbb{D}^{2,2}(w),
$$
and $\mathcal{B}[f](w,t)=\left[\mathcal{B}[f]^x(w,t),\mathcal{B}[f]^y(w,t) \right]$.
In \eqref{flux.an} we indicated with $\partial_x$, $\partial_y$ the partial derivatives in the directions defined by the components $w_x$ and $w_y$, respectively.  
Therefore the components of the flux $\mathcal{F}\textcolor{black}{(w,t)}$ are given by 
\begin{align}
\label{eq:Fx}
\mathcal{F}^x(w,t) &=\mathcal{C}^x(w,t)f(w,t)+\mathbb{D}^{1,1}(w)\partial_{x}f(w,t)+\mathbb{D}^{1,2}(w)\partial_{y}f(w,t),\\
\label{eq:Fy}
\mathcal{F}^y(w,t) &=\mathcal{C}^y(w,t)f(w,t)+\mathbb{D}^{2,1}(w)\partial_{x}f(w,t)+\mathbb{D}^{2,2}(w)\partial_{y}f(w,t).
\end{align}

\textcolor{black}{We want to approximate the stationary state, i.e. the multivariate distribution function $f^{\infty}(w)$ satisfying
\begin{equation}\label{st_state}
\nabla \cdot \mathcal{F}^{\infty}\textcolor{black}{(w)}=0,
\end{equation}
where $\mathcal{F}^{\infty}\textcolor{black}{(w)}=\mathcal{C}(w)f^{\infty}(w)+\mathbb{D}(w)\nabla_w f^{\infty}(w)$.  \textcolor{black}{We remark that now $\mathcal{C}=\mathcal{C}(w)$ depends only on $w$ as, in general, $\mathcal{C}$ depends on time only through $f$ and we are now considering the stationary state.}
A sufficient condition for $f^{\infty}\textcolor{black}{(w)}$ to be a stationary state is that it makes the flux vanish, i.e. 
\begin{equation}\label{st_state.nul_flux}
\mathcal{F}^{\infty}\textcolor{black}{(w)}=0
\end{equation}
that corresponds to have a constant $\mathcal{F}^{\infty}\textcolor{black}{(w)}$ that may be zero by fixing accordingly the no-flux boundary condition.
Equation \eqref{st_state.nul_flux} is a necessary condition only in one-dimension, while in a two-dimensional setting there may be stationary states satisfying \eqref{st_state} such that $\mathcal{F}^{\infty}\textcolor{black}{(w)}\neq 0$.
We concentrate on the problems in which \eqref{st_state.nul_flux} is satisfied.}

\textcolor{black}{
Let us observe that in 2D the condition $\mathcal{F}^{\infty}\textcolor{black}{(w)}=0$ defines a decoupled system of Fokker-Planck equations if $\mathbb{D}^{1,2}(w)=\mathbb{D}^{2,1}(w)=0$, $w\in \Omega$, that can be solved through standard schemes. In 1D, several numerical strategies to catch the emerging equilibrium have been designed. Schemes for Fokker-Planck-type equations have been  studied in the community: without intending to review the huge literature in this direction we mention schemes for linear drift-diffusion-type problems  \cite{BD,CC,LLPS,SG} together with related entropy methods \cite{BCDS,CHJS}, and recent developments for the general energy-decaying problems \cite{BCH,Pareschi_Zanella}. 
In particular, in \cite{Pareschi_Zanella} an analogous nonlinear Fokker-Planck equation with diagonal diffusion matrix was considered. In the present work, we will derive the scheme under the hypothesis in which the flux vanishes at the stationary state.
}

\textcolor{black}{Let us consider} $\mathbb D^{1,1}\textcolor{black}{(w)}, \mathbb D^{2,2}\textcolor{black}{(w)}\ne 0 \, \forall\textcolor{black}{w} \in \Omega$, we can define the following quasi-stationary system for the components of the flux
\begin{equation}\label{eq:conditions_SS}
\begin{split}
\mathbb D^{1,1}\textcolor{black}{(w)} \partial_x f\textcolor{black}{(w,t)} &= -f\textcolor{black}{(w,t)} \mathcal C^x- \mathbb D^{1,2}\textcolor{black}{(w)}\partial_y f\textcolor{black}{(w,t)}, \\
\mathbb D^{2,2}\textcolor{black}{(w)} \partial_y f\textcolor{black}{(w,t)} &= -f\textcolor{black}{(w,t)} \mathcal C^y - \mathbb D^{2,1}\textcolor{black}{(w)}\partial_x f\textcolor{black}{(w,t)}.
\end{split}
\end{equation}
\textcolor{black}{The latter system is quasi-stationary because $\mathcal{C}^x\textcolor{black}{(w,t)}$ and $\mathcal{C}^y\textcolor{black}{(w,t)}$ depend on $f\textcolor{black}{(w,t)}$.}
Let us observe that, thanks to the introduction of the matrix characterizing the diffusion the equations \eqref{eq:conditions_SS} are not decoupled unless $\mathbb D$ is diagonal. \textcolor{black}{From  \eqref{eq:conditions_SS}} we have 
\begin{equation}\label{sys.1}
\begin{split}
\left(\mathbb{D}^{1,1}\textcolor{black}{(w)}-\dfrac{\mathbb{D}^{1,2}\textcolor{black}{(w)}\mathbb{D}^{2,1}\textcolor{black}{(w)}}{\mathbb{D}^{2,2}\textcolor{black}{(w)}}\right) \partial_{x}f\textcolor{black}{(w,t)} &= -f\textcolor{black}{(w,t)}\left(\mathcal{C}^x\textcolor{black}{(w,t)}- \dfrac{\mathbb{D}^{1,2}\textcolor{black}{(w)}}{\mathbb{D}^{2,2}\textcolor{black}{(w)}}\mathcal{C}^y\textcolor{black}{(w,t)}\right), \\
\left(\mathbb{D}^{2,2}\textcolor{black}{(w)}-\dfrac{\mathbb{D}^{1,2}\textcolor{black}{(w)}\mathbb{D}^{2,1}\textcolor{black}{(w)}}{\mathbb{D}^{1,1}\textcolor{black}{(w)}} \right) \partial_{y}f\textcolor{black}{(w,t)} &= -f\textcolor{black}{(w,t)}\left(\mathcal{C}^y\textcolor{black}{(w,t)}-\dfrac{\mathbb{D}^{2,1}\textcolor{black}{(w)}}{\mathbb{D}^{1,1}\textcolor{black}{(w)}}\mathcal{C}^x\textcolor{black}{(w,t)}\right).
\end{split}
\end{equation}
In the following we will adopt the notations
\begin{equation}\label{D.call}
\mathcal{D}^1(w)=\mathbb{D}^{1,1}(w)-\dfrac{\mathbb{D}^{1,2}(w)\mathbb{D}^{2,1}(w)}{\mathbb{D}^{2,2}(w)},\qquad
\mathcal{D}^2(w)=\mathbb{D}^{2,2}(w)-\dfrac{\mathbb{D}^{1,2}(w)\mathbb{D}^{2,1}(w)}{\mathbb{D}^{1,1}(w)},
\end{equation}
\textcolor{black}{that are positive quantities since $\mathbb{D}$ is positive definite in the interior of $\Omega$.}

It is worth stressing how in the case $\mathbb D^{1,2}(w) = \mathbb D^{2,1}(w) = 0$, the two equations in \eqref{sys.1} can be decoupled and we basically recover the classical quasi-stationary formulation in each direction, we refer to  \cite{LLPS,Pareschi_Zanella} for more details on the concept of quasi-equilibrium distribution and to \cite{PR} for further applications.  Furthermore, we remark how system \eqref{sys.1} is in general not \textcolor{black}{analytically} solvable, except in some special cases due to the nonlinearity on the right hand side and the intrinsically  coupled nature of the system. We overcome this difficulty in the quasi steady-state approximation integrating the equations of system \eqref{sys.1} over numerical grids. 

\subsection{Derivation of the scheme}\label{derivation}
Let us consider $\Omega=[a,b]\times [a,b]$ and let us introduce the regular grid 
\[
 W=\left\lbrace w_{i,j}=(w_{x,i},w_{y,j}) \in \Omega | w_{x,i}=a+i\Delta w, w_{y,j}=a+j\Delta w, i,j=0,..,N, \Delta w=\frac{b-a}{N}\right\rbrace.
\] 
We shall also define the mid points grid as 
\[
\begin{split} 
W^{\textrm{mid}}= &\left\lbrace w_{i+ 1/2,j+1/2} =(w_{x,i+1/2},w_{y,j+1/2})\in \Omega | \right. \\
&\qquad\qquad \left. w_{x,i+1/2}=a+\frac{i\Delta w}{2}, w_{y,j+1/2}=a+\frac{j\Delta w}{2}, i,j=0,..,N-1\right\rbrace
\end{split}
\] 
and we remark that $W^{\textrm{mid}}$ is in the interior of $\Omega$.
Without loss of generality, and to avoid unnecessary complications, we consider a square domain with an uniform grid, i.e. with square cells; anyway what  presented in the following can be easily generalized to the case of rectangular cells, in which $w_{x,{i+1}}-w_{x,i}=\Delta w_1$ and $w_{y,{j+1}}-w_{y,j}=\Delta w_2$. Next, if we integrate the two equations in \eqref{sys.1} with respect to $w_x$ on $[w_{i,j},w_{i+1,j}]$ and with respect to $w_y$ on $[w_{i,j},w_{i,j+1}]$, respectively, we have
\begin{equation*}
\begin{split}
\displaystyle\int_{w_{i,j}}^{w_{i+1,j}} \dfrac{\partial_{x}f(w,t)}{f(w,t)} \, dw_x &=-\displaystyle\int_{w_{i,j}}^{w_{i+1,j}} \dfrac{1}{\mathcal D^1(w)} \left(\mathcal{C}^x(w,t)- \dfrac{\mathbb{D}^{1,2}(w)}{\mathbb{D}^{2,2}(w)}\mathcal{C}^y(w,t) \right) \,dw_x,\\
\displaystyle\int_{w_{i,j}}^{w_{i,j+1}} \dfrac{\partial_{y}f(w,t)}{f(w,t)} \, dw_y &=-\displaystyle\int_{w_{i,j}}^{w_{i,j+1}} \dfrac{1}{\mathcal D^2(w)} \left(\mathcal{C}^y(w,t)-\dfrac{\mathbb{D}^{2,1}(w)}{\mathbb{D}^{1,1}(w)}\mathcal{C}^x(w,t)\right) \, dw_y,
\end{split}
\end{equation*}
leading respectively to 
\begin{equation}\label{an.steady.state.1}
\displaystyle \dfrac{f(w_{i+1,j},t)}{f(w_{i,j},t)}=\exp \left\{ - \displaystyle\int_{w_{i,j}}^{w_{i+1,j}} \dfrac{1}{\mathcal D^1(w)} \left(\mathcal{C}^x(w,t)- \dfrac{\mathbb{D}^{1,2}(w)}{\mathbb{D}^{2,2}(w)}\mathcal{C}^y(w,t)\right) \,dw_x\right\}
\end{equation}
and
\begin{equation}\label{an.steady.state.2}
\displaystyle \dfrac{f(w_{i,j+1},t)}{f(w_{i,j},t)}=\exp \left\{ - \displaystyle\int_{w_{i,j}}^{w_{i,j+1}} \dfrac{1}{\mathcal D^2(w)}\left(\mathcal{C}^y(w,t)-\dfrac{\mathbb{D}^{2,1}(w)}{\mathbb{D}^{1,1}(w)}\mathcal{C}^x(w,t)\right) \, dw_y\right\}. 
\end{equation}

 Let us denote $f_{i,j}(t)$ an approximation of $f(w_{i,j},t)$ over the grid $W$. Let us introduce the following finite difference scheme where we keep the time continuous
\begin{equation}\label{eq:cons_disc}
\dfrac{d}{dt}f_{i,j}(t) = \dfrac{\mathcal{F}^{x}_{i+1/2,j}(t)-\mathcal{F}^{x}_{i-1/2,j}(t)}{\Delta w} + \dfrac{\mathcal{F}^{y}_{i,j+1/2}(t)-\mathcal{F}^{y}_{i,j-1/2}(t)}{\Delta w} ,
\end{equation}
where the right hand side is a numerical approximation of the operator $\nabla_w\cdot \mathcal{F}\textcolor{black}{(w,t)}$ on the grid $W$ at time $t>0$. The quantities $\mathcal{F}^{x}_{i\pm 1/2,j}\textcolor{black}{(t)}$, $\mathcal{F}^{y}_{i,j\pm 1/2}\textcolor{black}{(t)}$ are the numerical flux functions relative to the introduced numerical discretization. We want to define the numerical fluxes analogously to \cite{Pareschi_Zanella}, where they give a second order definition for the two components of the numerical flux, i.e. $\mathcal{F}^{x}_{i+1/2,j}\textcolor{black}{(t)}$  and $\mathcal{F}^{y}_{i,j+1/2}\textcolor{black}{(t)}$ are combinations of the grid points $i+1$ and $i$, $j+1$ and $j$ respectively. In the rest of this section we will omit the explicit dependency on time. 
\\

In particular, in \cite{Pareschi_Zanella}, where $\mathbb{D}^{1,2}(w)=\mathbb{D}^{2,1}(w)=0$, the authors define the numerical fluxes as 
\begin{equation}\label{eq:fluxes_partial.old}
\begin{split}
\mathcal{F}^{x}_{i+1/2,j}\textcolor{black}{(t)} &=\tilde{\mathcal{C}}^{x}_{i+1/2,j}\textcolor{black}{(t)}\tilde{f}_{i+1/2,j}\textcolor{black}{(t)}+\mathbb{D}^{1,1}_{i+1/2,j} \dfrac{f_{i+1,j}\textcolor{black}{(t)}-f_{i,j}\textcolor{black}{(t)}}{\Delta w}, \\
\mathcal{F}^{y}_{i,j+1/2}\textcolor{black}{(t)} &=\tilde{\mathcal{C}}^{y}_{i,j+1/2}\textcolor{black}{(t)}\tilde{f}_{i,j+1/2}\textcolor{black}{(t)}+\mathbb{D}^{2,2}_{i,j+1/2} \dfrac{f_{i,j+1}\textcolor{black}{(t)}-f_{i,j}\textcolor{black}{(t)}}{\Delta w},
\end{split}
\end{equation}
where $\tilde{f}_{i+1/2,j}\textcolor{black}{(t)},\tilde{f}_{i,j+1/2}\textcolor{black}{(t)}$ are classically defined as 
\begin{equation}\label{f.tilde}
\begin{split}
\tilde{f}_{i+1/2,j}\textcolor{black}{(t)}=(1-\delta_{i+1/2,j}\textcolor{black}{(t)})f_{i+1,j}\textcolor{black}{(t)}+\delta_{i+1/2,j}\textcolor{black}{(t)}f_{i,j}\textcolor{black}{(t)} ,\\
 \tilde{f}_{i,j+1/2}\textcolor{black}{(t)}=(1-\delta_{i,j+1/2}\textcolor{black}{(t)})f_{i,j+1}\textcolor{black}{(t)}+\delta_{i,j+1/2}\textcolor{black}{(t)}f_{i,j}\textcolor{black}{(t)},
 \end{split}
\end{equation}
see \cite{CC, MB, Pareschi_Zanella}. The weight functions $\delta_{i+1/2,j}\textcolor{black}{(t)}$, $\delta_{i,j+1/2}\textcolor{black}{(t)}$ are hence defined in such a way that they have values in $(0,1)$  and, thus, $\tilde{f}_{i+1/2,j}\textcolor{black}{(t)}$ and $\tilde{f}_{i,j+1/2}\textcolor{black}{(t)}$ are convex combinations of $f_{i+1,j}\textcolor{black}{(t)}, f_{i,j}\textcolor{black}{(t)}$ and $f_{i,j+1}\textcolor{black}{(t)},f_{i+1,j}\textcolor{black}{(t)}$ respectively. 
\\
In the present setting, since the extra diagonal terms of the diffusion matrix do not vanish, the definition of the numerical fluxes must be modified accordingly. In particular, we shall write as an extension of \eqref{eq:fluxes_partial.old} 
\begin{equation}\label{eq:fluxes_partial}
\begin{split}
\mathcal{F}^{x}_{i+1/2,j}\textcolor{black}{(t)} &=\tilde{\mathcal{C}}^{x}_{i+1/2,j}\textcolor{black}{(t)}\tilde{f}_{i+1/2,j}\textcolor{black}{(t)}+\mathbb{D}^{1,1}_{i+1/2,j} \dfrac{f_{i+1,j}\textcolor{black}{(t)}-f_{i,j}\textcolor{black}{(t)}}{\Delta w}+\mathbb{D}^{1,2}_{i+1/2,j}[\partial_{y}f]_{i,j}\textcolor{black}{(t)}, \\
\mathcal{F}^{y}_{i,j+1/2}\textcolor{black}{(t)} &=\tilde{\mathcal{C}}^{y}_{i,j+1/2}\textcolor{black}{(t)}\tilde{f}_{i,j+1/2}\textcolor{black}{(t)}+\mathbb{D}^{2,2}_{i,j+1/2} \dfrac{f_{i,j+1}\textcolor{black}{(t)}-f_{i,j}\textcolor{black}{(t)}}{\Delta w}+\mathbb{D}^{2,1}_{i,j+1/2}[\partial_{x}f]_{i,j}\textcolor{black}{(t)},
\end{split}
\end{equation}
where $[\partial_{y}f]_{i,j}\textcolor{black}{(t)}$ and $[\partial_{x}f]_{i,j}\textcolor{black}{(t)}$ are numerical approximations of the partial derivatives $\partial_{y}f(w,t)$ and $\partial_{x}f(w,t)$ \textcolor{black}{that we need to determine}. As we want to perform a directional splitting,  we have to determine the approximations $[\partial_{y}f]_{i,j}\textcolor{black}{(t)}$ and $[\partial_{x}f]_{i,j}\textcolor{black}{(t)}$  in the complementary direction with respect to the one of the differentiation, i.e. as a combination of $f_{i+1,j}\textcolor{black}{(t)}, f_{i,j}\textcolor{black}{(t)}$ and $f_{i,j+1}\textcolor{black}{(t)}, f_{i,j}\textcolor{black}{(t)}$ respectively.
In order to obtain such approximations, in addition to $\mathcal{F}^{x}_{i+1/2,j}\textcolor{black}{(t)} = 0$ and $\mathcal{F}^{y}_{i,j+1/2}\textcolor{black}{(t)} = 0$, we consider the discretization of the two components of the numerical fluxes in the complementary direction, $\ie$ we discretize $\mathcal{F}^x\textcolor{black}{(w,t)}$ in the $y$ direction and $\mathcal{F}^y\textcolor{black}{(w,t)}$  in the $x$ direction:
\begin{equation}\label{eq:flux_compl}
\begin{split}
\mathcal{F}^{x}_{i,j+1/2}\textcolor{black}{(t)} &=\tilde{\mathcal{C}}^{x}_{i,j+1/2}\textcolor{black}{(t)}\tilde{f}_{i,j+1/2}\textcolor{black}{(t)}+\mathbb{D}^{1,2}_{i,j+1/2} \dfrac{f_{i,j+1}\textcolor{black}{(t)}-f_{i,j}\textcolor{black}{(t)}}{\Delta w}+\mathbb{D}^{1,1}_{i,j+1/2}[\partial_{x}f]_{i,j}\textcolor{black}{(t)}, \\
\mathcal{F}^{y}_{i+1/2,j}\textcolor{black}{(t)} &=\tilde{\mathcal{C}}^{y}_{i+1/2,j}\textcolor{black}{(t)}\tilde{f}_{i+1/2,j}\textcolor{black}{(t)}+\mathbb{D}^{2,1}_{i+1/2,j} \dfrac{f_{i+1,j}\textcolor{black}{(t)}-f_{i,j}\textcolor{black}{(t)}}{\Delta w}+\mathbb{D}^{2,2}_{i+1/2,j}[\partial_{y}f]_{i,j}\textcolor{black}{(t)}.
\end{split}
\end{equation}
By making the latter vanish, i.e. $\mathcal{F}^{x}_{i,j+1/2}\textcolor{black}{(t)} = 0$ and $\mathcal{F}^{y}_{i+1/2,j}\textcolor{black}{(t)} = 0$, we find the following numerical approximations in $w_{i,j}$ of the partial derivatives $[\partial_{y}f]_{i,j}\textcolor{black}{(t)}$ and $[\partial_{x}f]_{i,j}\textcolor{black}{(t)}$ in the complementary direction with respect to the one of the differentiation
\begin{equation}\label{eq:syst1.r}
[\partial_{y}f]_{i,j}\textcolor{black}{(t)}=-\dfrac{1}{\mathbb{D}^{2,2}_{i+1/2,j}}\left[ \tilde{\mathcal{C}}^{y}_{i+1/2,j}\textcolor{black}{(t)}\tilde{f}_{i+1/2,j}\textcolor{black}{(t)}+\mathbb{D}^{2,1}_{i+1/2,j} \dfrac{f_{i+1,j}\textcolor{black}{(t)}-f_{i,j}\textcolor{black}{(t)}}{\Delta w}\right],
\end{equation}
and
\begin{equation}\label{eq:syst2.r}
[\partial_{x}f]_{i,j}\textcolor{black}{(t)}=-\dfrac{1}{\mathbb{D}^{1,1}_{i,j+1/2}}\left[ \tilde{\mathcal{C}}^{x}_{i,j+1/2}\textcolor{black}{(t)}\tilde{f}_{i,j+1/2}\textcolor{black}{(t)}+\mathbb{D}^{1,2}_{i,j+1/2} \dfrac{f_{i,j+1}\textcolor{black}{(t)}-f_{i,j}\textcolor{black}{(t)}}{\Delta w}\right],
\end{equation}
where $\tilde{f}_{i+1/2,j}\textcolor{black}{(t)}, \tilde{f}_{i,j+1/2}\textcolor{black}{(t)}$ are given by \eqref{f.tilde}.
By substituting \eqref{eq:syst1.r} and \eqref{eq:syst2.r} in Eq \eqref{eq:fluxes_partial} we obtain
\begin{subequations} \label{eq:flux_th}
\begin{equation}\label{eq:flux_th.a}
\mathcal{F}^{x}_{i+1/2,j}\textcolor{black}{(t)} = \tilde{\mathcal{G}}^{x}_{i+1/2,j}\textcolor{black}{(t)} \tilde f_{i+1/2,j}\textcolor{black}{(t)} + \mathcal D^1_{i+1/2,j} \dfrac{f_{i+1,j}\textcolor{black}{(t)}-f_{i,j}\textcolor{black}{(t)}}{\Delta w},
\end{equation}\\
\begin{equation}\label{eq:flux_th.b}
\mathcal{F}^{y}_{i,j+1/2}\textcolor{black}{(t)} = \tilde{\mathcal{G}}^{y}_{i,j+1/2}\textcolor{black}{(t)} \tilde f_{i,j+1/2}\textcolor{black}{(t)} + \mathcal D^2_{i,j+1/2} \dfrac{f_{i,j+1}\textcolor{black}{(t)}-f_{i,j}\textcolor{black}{(t)}}{\Delta w},
\end{equation}
\end{subequations}
where $\tilde f_{i+1/2,j}\textcolor{black}{(t)}$, $\tilde f_{i,j+1/2}\textcolor{black}{(t)}$ are expressed as in \eqref{f.tilde} and
\begin{equation}\label{Ltilde}
\begin{split}
\tilde{\mathcal{G}}^{x}_{i+1/2,j}\textcolor{black}{(t)}=\tilde{\mathcal{C}}^{x}_{i+1/2,j}\textcolor{black}{(t)}-\dfrac{\mathbb{D}^{1,2}_{i+1/2,j}}{\mathbb{D}^{2,2}_{i+1/2,j}}\tilde{\mathcal{C}}^{y}_{i+1/2,j}\textcolor{black}{(t)}, \\
\tilde{\mathcal{G}}^{y}_{i,j+1/2}\textcolor{black}{(t)}=\tilde{\mathcal{C}}^{y}_{i,j+1/2}\textcolor{black}{(t)}-\dfrac{\mathbb{D}^{2,1}_{i,j+1/2}}{\mathbb{D}^{1,1}_{i,j+1/2}}\tilde{\mathcal{C}}^{x}_{i,j+1/2}\textcolor{black}{(t)}. 
\end{split}
\end{equation}
We shall now equate to zero the two components of the numerical flux \eqref{eq:flux_th}. By setting  \eqref{eq:flux_th.a} to zero, where $\tilde{f}_{i+1/2,j}\textcolor{black}{(t)}$ is defined as in \eqref{f.tilde} and $\tilde{\mathcal{G}}^{x}_{i+1/2,j}\textcolor{black}{(t)}$ as in \eqref{Ltilde}, we obtain 
\[
f_{i+1,j}\textcolor{black}{(t)}(1-\delta_{i+1/2,j}\textcolor{black}{(t)})\tilde{\mathcal{G}}^{x}_{i+1/2,j}\textcolor{black}{(t)}+\dfrac{\mathcal D^1_{i+1/2,j}}{\Delta w}+f_{i,j}\textcolor{black}{(t)}\delta_{i+1/2,j}\textcolor{black}{(t)}\tilde{\mathcal{G}}^{x}_{i+1/2,j}\textcolor{black}{(t)}+\dfrac{\mathcal D^1_{i+1/2,j}}{\Delta w}=0
\]
and, therefore
\begin{equation}\label{num.steady.state.1}
\dfrac{f_{i+1,j}\textcolor{black}{(t)}}{f_{i,j}\textcolor{black}{(t)}}=\dfrac{-\delta_{i+1/2,j}\textcolor{black}{(t)}\tilde{\mathcal{G}}^{x}_{i+1/2,j}\textcolor{black}{(t)}+\dfrac{\mathcal D^1_{i+1/2,j}}{\Delta w}}{(1-\delta_{i+1/2,j}\textcolor{black}{(t)})\tilde{\mathcal{G}}^{x}_{i+1/2,j}\textcolor{black}{(t)}+\dfrac{\mathcal D^1_{i+1/2,j}}{\Delta w}}.
\end{equation}
Analogously, equating \eqref{eq:flux_th.b} to zero gives 
\begin{equation}\label{num.steady.state.2}
\dfrac{f_{i,j+1}\textcolor{black}{(t)}}{f_{i,j}\textcolor{black}{(t)}}=\dfrac{-\delta_{i,j+1/2}\textcolor{black}{(t)}\tilde{\mathcal{G}}^{y}_{i,j+1/2}\textcolor{black}{(t)}+\dfrac{\mathcal D^2_{i,j+1/2}}{\Delta w}}{(1-\delta_{i,j+1/2}\textcolor{black}{(t)})\tilde{\mathcal{G}}^{y}_{i,j+1/2}\textcolor{black}{(t)}+\dfrac{\mathcal D^2_{i,j+1/2}}{\Delta w}},
\end{equation}
where, as a consequence of the definition \eqref{D.call}, we have
\[
\begin{split}
\mathcal{D}^1_{i+1/2,j}=\mathbb{D}^{1,1}(w_{i+1/2,j})-\dfrac{\mathbb{D}^{1,2}(w_{i+1/2,j})\mathbb{D}^{2,1}(w_{i+1/2,j})}{\mathbb{D}^{2,2}(w_{i+1/2,j})}>0,\\
\mathcal{D}^2_{i,j+1/2}=\mathbb{D}^{2,2}(w_{i,j+1/2})-\dfrac{\mathbb{D}^{1,2}(w_{i,j+1/2})\mathbb{D}^{2,1}(w_{i,j+1/2})}{\mathbb{D}^{1,1}(w_{i,j+1/2})}>0.
\end{split}
\]

We now need to define suitable weight functions $\delta_{i+1/2,j}\textcolor{black}{(t)}$, $\delta_{i,j+1/2}\textcolor{black}{(t)}$ and numerical drifts $\tilde{\mathcal C}^{x}\textcolor{black}{(w,t)}$, $\tilde{\mathcal C}^{y}\textcolor{black}{(w,t)}$ so that the method preserves the steady state of the problem with arbitrary accuracy and so that its numerical solution defines nonnegative solutions without additional restrictions on the grid $\Delta w$. 
By equating analytical and the numerical form of the flux, i.e. $f(w_{i+1,j},t)/f(w_{i,j},t)$ in \eqref{an.steady.state.1} with $f_{i+1,j}\textcolor{black}{(t)}/f_{i,j}\textcolor{black}{(t)}$ in \eqref{num.steady.state.1}, and $f(w_{i,j+1},t)/f(w_{i,j},t)$ in \eqref{an.steady.state.2} with $f_{i,j+1}\textcolor{black}{(t)}/f_{i,j}\textcolor{black}{(t)}$ in \eqref{num.steady.state.2}, and setting
\[
\begin{split}
\tilde{\mathcal{C}}^{x}_{i+1/2,j}\textcolor{black}{(t)}   &= \dfrac{\mathcal D^1_{i+1/2,j}}{\Delta w}\displaystyle\int_{w_{i,j}}^{w_{i+1,j}} \dfrac{\mathcal{C}^x(w,t)}{\mathcal{D}^1(w)} \, dw_x, \\
\tilde{\mathcal{C}}^{y}_{i+1/2,j}\textcolor{black}{(t)}   &=\dfrac{\mathcal D^1_{i+1/2,j}}{\Delta w} \displaystyle\int_{w_{i,j}}^{w_{i+1,j}} \dfrac{\mathcal{C}^y(w,t)}{\mathcal{D}^1(w)} \, dw_x,
\end{split}
\]
and 
\[
\begin{split}
\tilde{\mathcal{C}}^{x}_{i,j+1/2}\textcolor{black}{(t)}  & = \dfrac{\mathcal D^2_{i,j+1/2}}{\Delta w}  \displaystyle\int_{w_{i,j}}^{w_{i,j+1}} \dfrac{\mathcal{C}^x(w,t)}{\mathcal{D}^2(w)} \, dw_y,\\
\tilde{\mathcal{C}}^{y}_{i,j+1/2}\textcolor{black}{(t)}  &=\dfrac{\mathcal D^2_{i,j+1/2}}{\Delta w}\displaystyle\int_{w_{i,j}}^{w_{i,j+1}} \dfrac{\mathcal{C}^y(w,t)}{\mathcal{D}^2(w)}\, dw_y,
\end{split}
\]
we finally get 
\begin{equation}\label{delta}
\begin{split}
\delta_{i+1/2,j}\textcolor{black}{(t)}  =\dfrac{1}{\lambda_{i+1/2,j}\textcolor{black}{(t)}}+\dfrac{1}{1-\exp (\lambda_{i+1/2,j}\textcolor{black}{(t)})}, \qquad  \delta_{i,j+1/2}\textcolor{black}{(t)} =\dfrac{1}{\lambda_{i,j+1/2}\textcolor{black}{(t)}}+\dfrac{1}{1-\exp (\lambda_{i,j+1/2}\textcolor{black}{(t)})},
\end{split}
\end{equation}
where
\begin{equation}\label{lambda}
\begin{split}
\lambda_{i+1/2,j}\textcolor{black}{(t)}= \displaystyle\int_{w_{i,j}}^{w_{i+1,j}}\dfrac{1}{\mathcal D^1(w)} \left( \mathcal{C}^x(w,t)- \dfrac{\mathbb{D}^{1,2}\textcolor{black}{(w)}}{\mathbb{D}^{2,2}\textcolor{black}{(w)}}\mathcal{C}^y(w,t)\right) \,dw_x=\dfrac{\Delta w}{\mathcal D^1_{i+1/2,j}} \tilde{\mathcal{G}}^{x}_{i+1/2,j}\textcolor{black}{(t)},\\
\lambda_{i,j+1/2}\textcolor{black}{(t)}=\displaystyle\int_{w_{i,j}}^{w_{i,j+1}}\dfrac{1}{\mathcal D^2(w)} \left( \mathcal{C}^y(w,t)-\dfrac{\mathbb{D}^{2,1}\textcolor{black}{(w)}}{\mathbb{D}^{1,1}\textcolor{black}{(w)}}\mathcal{C}^x(w,t) \right) \, dw_y=\dfrac{\Delta w}{\mathcal D^2_{i,j+1/2}} \tilde{\mathcal{G}}^{y}_{i,j+1/2}\textcolor{black}{(t)}.
\end{split}
\end{equation}

We have the following result

\begin{theorem}\label{th1}
The numerical flux defined by \eqref{eq:fluxes_partial} with \eqref{eq:syst1.r}-\eqref{eq:syst2.r} is given by \eqref{eq:flux_th}
with $\tilde{\mathcal{G}}^{x}_{i+1/2,j}\textcolor{black}{(t)}$, $\tilde{\mathcal{G}}^{y}_{i,j+1/2}\textcolor{black}{(t)}$ defined in \eqref{Ltilde} and with $\delta_{i+1/2,j}\textcolor{black}{(t)}$, $\delta_{i,j+1/2}\textcolor{black}{(t)}$ defined in \eqref{delta}. The  numerical flux \eqref{eq:flux_th} vanishes when the flux \eqref{eq:Fx}-\eqref{eq:Fy} \textcolor{black}{vanishes} over the cell $[w_{i,j},w_{i+1,j}] \times [w_{i,j},w_{i,j+1}]$. The nonlinear weights defined in \eqref{delta}-\eqref{lambda} are such that $\delta_{i\pm 1/2,j}\textcolor{black}{(t)} \in (0,1)$, $\delta_{i,j\pm 1/2}\textcolor{black}{(t)}\in(0,1)$. 
\end{theorem}
\begin{proof}
The form of the flux comes from the computations present in this section. 
If we equate \eqref{eq:fluxes_partial} to zero we can guarantee that the exact flux vanishes in the derived numerical approximation in the case where the components of the analytical flux vanish in the presence of a steady state. Finally, the latter result follows from the inequality $\exp(x) \ge 1+x$. 
\end{proof}

\begin{remark}
We can observe that for $\lambda_{i+1/2,j}\textcolor{black}{(t)}\ll 1$ and  $\lambda_{i,j+1/2}\textcolor{black}{(t)}\ll 1$ we have
\[
\delta_{i+1/2,j}\textcolor{black}{(t)} = \dfrac{1}{2} + O(\lambda_{i+1/2,j}\textcolor{black}{(t)}), \delta_{i,j+1/2}\textcolor{black}{(t)} = \dfrac{1}{2} + O(\lambda_{i,j+1/2}\textcolor{black}{(t)}),
\]
and, therefore, when $\lambda_{i+1/2,j}\textcolor{black}{(t)}=\lambda_{i,j+1/2}\textcolor{black}{(t)}=0$, we have that $\delta_{i+1/2,j}\textcolor{black}{(t)}=\textcolor{black}{\delta_{i,j+1/2}}\textcolor{black}{(t)}=\dfrac{1}{2}$. 
\end{remark}

\textcolor{black}{In the scheme defined by the numerical fluxes \eqref{eq:flux_th}
with $\tilde{\mathcal{G}}^{x}_{i+1/2,j}\textcolor{black}{(t)}$ and $\tilde{\mathcal{G}}^{y}_{i,j+1/2}\textcolor{black}{(t)}$ defined in \eqref{Ltilde} and with $\delta_{i+1/2,j}\textcolor{black}{(t)}$, $\delta_{i,j+1/2}\textcolor{black}{(t)}$ defined by \eqref{delta}-\eqref{lambda}, we shall approximate the integrals appearing in the nonlinear weights \eqref{lambda} through high order quadrature rules, as done in \cite{Pareschi_Zanella}.
In fact, it is worth observing that the derived scheme may be seen as a generalization of the classical second-order Chang-Cooper scheme \cite{CC,LLPS} to \textcolor{black}{Fokker-Planck equations with general diffusion matrix}.} In their original formulation, these works focused on linear Fokker-Planck equations with diagonal diffusion matrix and a recent generalization to the nonlinear case has been proposed in \cite{Pareschi_Zanella}. We highlight how the scheme proposed in \cite{Pareschi_Zanella} and the present scheme are coherent to the original one by approximating the functions  \eqref{lambda} through a midpoint quadrature rule as follows
\begin{equation*}
\begin{split}
&\lambda_{i+1/2,j}^{\textrm{mid}}\textcolor{black}{(t)} = \displaystyle\int_{w_{i,j}}^{w_{i+1,j}}\dfrac{1}{\mathcal D^1(w)} \left(\mathcal{C}^x(w,t)- \dfrac{\mathbb{D}^{1,2}\textcolor{black}{(w)}}{\mathbb{D}^{2,2}\textcolor{black}{(w)}}\mathcal{C}^y(w,t)\right) \,dw_x  \\
&\qquad\quad\;=  \dfrac{\Delta w}{\mathcal D^1_{i+1/2,j}}  \left(\mathcal{C}^{x}_{i+1/2,j}-\dfrac{\mathbb{D}^{1,2}_{i+1/2,j}}{\mathbb{D}^{2,2}_{i+1/2,j}}\mathcal{C}^{y}_{i+1/2,j}\right), \\
&\lambda_{i,j+1/2}^{\textrm{mid}}\textcolor{black}{(t)} = \displaystyle\int_{w_{i,j}}^{w_{i,j+1}} \dfrac{1}{\mathcal D^2(w)} \left(\mathcal{C}^y(w,t)-\dfrac{\mathbb{D}^{2,1}\textcolor{black}{(w)}}{\mathbb{D}^{1,1}\textcolor{black}{(w)}}\mathcal{C}^x(w,t)\right) \, dw_y  \\
&\qquad\quad\;= \dfrac{\Delta w}{\mathcal D^2_{i,j+1/2}}\left(\mathcal{C}^{y}_{i,j+1/2}-\dfrac{\mathbb{D}^{2,1}_{i,j+1/2}}{\mathbb{D}^{1,1}_{i,j+1/2}} \mathcal{C}^{x}_{i,j+1/2}\right),
\end{split}
\end{equation*}
leading to the following weights
\[
\begin{split}
\delta_{i+1/2,j}^{\textrm{mid}}\textcolor{black}{(t)} &=\dfrac{\mathcal D^1_{i+1/2,j}}{\Delta w\left(\mathcal{C}^x_{i+1/2,j}\textcolor{black}{(t)}-\dfrac{\mathbb{D}^{1,2}_{i+1/2,j}}{\mathbb{D}^{2,2}_{i+1/2,j}}\mathcal{C}^y_{i+1/2,j}\textcolor{black}{(t)}\right)}+\dfrac{1}{1-\exp (\lambda^{\textrm{mid}}_{i+1/2,j}\textcolor{black}{(t)})}, \\ 
\delta_{i,j+1/2}^{\textrm{mid}}\textcolor{black}{(t)} &=\dfrac{\mathcal D^2_{i,j+1/2}}{\Delta w\left(\mathcal{C}^y_{i,j+1/2}\textcolor{black}{(t)}-\dfrac{\mathbb{D}^{2,1}_{i,j+1/2}}{\mathbb{D}^{1,1}_{i,j+1/2}}\mathcal{C}^x_{i,j+1/2}\textcolor{black}{(t)}\right)}+\dfrac{1}{1-\exp (\lambda^{\textrm{mid}}_{i,j+1/2}\textcolor{black}{(t)})}.
\end{split}
\]
Hence, in the case $\mathbb D^{1,2}\textcolor{black}{(w)} = \mathbb D^{2,1}\textcolor{black}{(w)} = 0$ we recover the classical formulation. Furthermore, we observe that if $\mathcal{B}[f](w,t)$ does not depend on $f\textcolor{black}{(w,t)}$ and %levato =B(w)%
 has components which are first order polynomials, the midpoint rule gives an exact evaluation of the integrals in \eqref{lambda}. \textcolor{black}{More generally, to extend the introduced approach as in \cite{Pareschi_Zanella} we may consider standard high order quadrature rules for the computation of the nonlinear weights \eqref{lambda}, see e.g. \cite{Dahl}. }

\begin{remark}
\textcolor{black}{In the case $\mathcal{B}[f](w,t)=B(w)$ the quasi-stationary formulation \eqref{eq:conditions_SS} becomes stationary, because $\mathcal{C}\textcolor{black}{(w)}$ does not depend on $f\textcolor{black}{(w,t)}$ anymore.} Once we know the stationary state $f^{\infty}(w)$, we can compute the weights $\delta_{i+1/2,j}\textcolor{black}{(t)}$, $\delta_{i,j+1/2}\textcolor{black}{(t)}$ exactly. In fact, we have
\[
\begin{split}
\displaystyle \dfrac{f^{\infty}_{i+1,j}}{f^{\infty}_{i,j}} &=\exp \left\{ - \displaystyle\int_{w_{i,j}}^{w_{i+1,j}} \dfrac{1}{\mathcal D^1(w)}\left(\mathcal{C}^x(w,t)- \dfrac{\mathbb{D}^{1,2}\textcolor{black}{(w)}}{\mathbb{D}^{2,2}\textcolor{black}{(w)}}\mathcal{C}^y(w,t)\right) \,dw_x\right\}  \\
&=\exp\left\{-\lambda^{\infty}_{i+1/2,j}\right\} \\
\displaystyle \dfrac{f^{\infty}_{i,j+1}}{f^{\infty}_{i,j}} &=\exp \left\{ - \displaystyle\int_{w_{i,j}}^{w_{i,j+1}} \dfrac{1}{\mathcal D^2(w)} \left(\mathcal{C}^y(w,t)-\dfrac{\mathbb{D}^{2,1}\textcolor{black}{(w)}}{\mathbb{D}^{1,1}\textcolor{black}{(w)}}\mathcal{C}^x(w,t)\right) \, dw_y\right\} \\
& =\exp\left\{-\lambda^{\infty}_{i,j+1/2}\right\},
\end{split}
\]
that define the following weights
\begin{equation}\label{delta.inf}
\begin{split}
\delta_{i+1/2,j}^\infty &=\dfrac{1}{\log f^{\infty}_{i,j} -\log f^{\infty}_{i+1,j}}+\dfrac{f^{\infty}_{i+1,j}}{f^{\infty}_{i+1,j}-f^{\infty}_{i,j}},\\
\delta_{i,j+1/2}^\infty &=\dfrac{1}{\log f^{\infty}_{i,j} -\log f^{\infty}_{i,j+1}}+\dfrac{f^{\infty}_{i,j+1}}{f^{\infty}_{i,j+1}-f^{\infty}_{i,j}}.
\end{split}
\end{equation}
\textcolor{black}{Using classical methods, as it is done for example in \cite{MB} for the linear Fokker-Planck equation with diagonal diffusion matrix, we can prove that the proposed scheme is consistent if the problem is linear and the flux vanishes at the steady state. In particular, using the same arguments of \cite{MB}, it is possible to see that the stationary state is approximated with an order equal to the order of the quadrature rule. This is cannot be proved for problems whose stationary state does not make the flux vanish (see Test2).}
\end{remark}

 \begin{remark}
 If we consider the limit case in which the diffusion tensor tends to be singular and the elements of $\nabla \cdot \mathbb{D}$ tend to vanish, we obtain
\[
\delta_{i+1/2,j}\textcolor{black}{(t)}=\\
\begin{cases}
0, \qquad \mathcal{B}_{i+1/2,j}\textcolor{black}{(t)}>0,\\
1, \qquad \mathcal{B}_{i+1/2,j}\textcolor{black}{(t)}<0,
\end{cases}
\qquad
\delta_{i,j+1/2}\textcolor{black}{(t)} = 
\begin{cases}
0 \qquad \mathcal B_{i,j+1/2}\textcolor{black}{(t)}>0, \\
1 \qquad \mathcal B_{i,j+1/2}\textcolor{black}{(t)}<0.
\end{cases}
\]
Therefore the scheme reduces to a first order upwind scheme. 
\end{remark}

\section{Main properties}\label{sect:prop}
In this section we show the properties of the derived numerical schemes. In particular, we will prove how the present method enforces conservation of mass, nonnegativity of the numerical solution and correctly dissipates the entropy. 

\textcolor{black}{
\subsection{Conservation of mass}
We notice that the no-flux boundary conditions
\[
\mathcal{F}\textcolor{black}{(w,t)}\cdot \boldsymbol{n}(w)=0, \quad w \in \partial \Omega
\]
ensure the conservation of mass in the problem \eqref{FP1} since 
$$\int_{\Omega} f(w,t) dw= \int_{\Omega} f_0(w) \, dw$$
for all time $t \ge 0$. At the numerical level, no-flux boundary conditions are obtained by imposing
\begin{equation}\label{no_flux.num}
\mathcal{F}^{x}_{N+1/2,j}\textcolor{black}{(t)} = \mathcal{F}^{x}_{-1/2,j}\textcolor{black}{(t)} = 0,\quad  \textrm{and} \quad \mathcal{F}^{y}_{i,N+1/2}\textcolor{black}{(t)} = \mathcal{F}^{y}_{i,-1/2}\textcolor{black}{(t)} = 0, \quad \forall i,j = 0,\dots,N, \quad t\ge 0, 
\end{equation}
and we can prove that the introduced scheme ensures the conservation of mass. 
\begin{lemma}
Let us consider a discretization of the problem \eqref{FP1} in the form \eqref{eq:cons_disc} complemented with no-flux boundary conditions \eqref{no_flux.num}. Then we have
\[
\dfrac{d}{dt} \sum_{i=0}^N \sum_{j=0}^N f_{i,j}(t) = 0.
\]
\end{lemma}
\begin{proof}
From \eqref{eq:cons_disc} we have
\[\begin{split}
& \sum_{i=0}^N \sum_{j=0}^N \dfrac{d}{dt} f_{i,j}\textcolor{black}{(t)} =\dfrac{1}{\Delta w} \sum_{j=0}^N \left( \mathcal{F}^{x}_{-1/2,j}\textcolor{black}{(t)}-\mathcal{F}^{x}_{N+1/2,j}\textcolor{black}{(t)} \right) +  \dfrac{1}{\Delta w} \sum_{i=0}^N \left( \mathcal{F}^{y}_{i,-1/2}\textcolor{black}{(t)}-\mathcal{F}^{y}_{j,N+1/2}\textcolor{black}{(t)} \right),
\end{split}\]
from which we conclude using \eqref{no_flux.num}. 
\end{proof}}
\subsection{Positivity \textcolor{black}{ of the explicit scheme}}

In this section we will provide results for non-negativity of the scheme with explicit time integration. \textcolor{black}{We introduce the time discretization $t^n = n\Delta t$,  $n = 0,\dots,N_T$ with $\Delta t = T/N_T$ being $T$ the final time}. We first consider the simple forward Euler method
\[
f_{i,j}^{n+1}=f_{i,j}^n+\Delta t\dfrac{\mathcal{F}^{x,n}_{i+1/2,j}-\mathcal{F}^{x,n}_{i-1/2,j}}{\Delta w} +\Delta t \dfrac{\mathcal{F}^{y,n}_{i,j+1/2}-\mathcal{F}^{y,n}_{i,j-1/2}}{\Delta w},
\]
where \textcolor{black}{$f_{i,j}^n = f_{i,j}(t^n)$} and $\mathcal{F}^{x,n}_{i+1/2,j}, \mathcal{F}^{y,n}_{i,j+1/2}$ are the numerical fluxes at time $t^n$, \textcolor{black}{i.e $\mathcal{F}^{x,n}_{i+1/2,j}=\mathcal{F}^{x}_{i+1/2,j}(t^n)$, and $\mathcal{F}^{y,n}_{i,j+1/2}=\mathcal{F}^{y}_{i,j+1/2}(t^n)$.}

We can prove the following result
\begin{theorem}
Under the time step restriction
\begin{equation}\label{pos.eu}
\Delta t \leq \dfrac{\Delta w^2}{2\left[(G_x + G_y)\Delta w+(D^1 + D^2)\right]}
\end{equation}
where 
$$
G_x=\max_{i,j,n} |\tilde{\mathcal{G}}^{x,n}_{i+1/2,j}|, \qquad G_y=\max_{i,j,n}|\tilde{\mathcal{G}}^{y,n}_{i,j+1/2}|,
$$
and 
$$D^1=\max_{i,j} \mathcal D^1_{i+1/2,j}, \qquad D^2=\max_{i,j} \mathcal D^2_{i,j+1/2}, $$ 
the explicit scheme preserves nonnegativity, $\ie$ $f_{i,j}^{n+1}\geq 0$ if $f_{i,j}^n \geq 0$.
\end{theorem}
\begin{proof}
We will adopt the structure of the scheme introduced in Theorem \ref{th1}. In details, the scheme reads
\[
\begin{split}
f_{i,j}^{n+1} &=f_{i,j}^n+\dfrac{\Delta t}{\Delta w}\left [ \left( \tilde{\mathcal{G}}^{x,n}_{i+1/2,j}(1-\delta_{i+1/2,j}^n)+\dfrac{\mathcal D^1_{i+1/2,j}}{\Delta w}\right)  f_{i+1,j}^n \right. \\ 
&- \left. \left(- \tilde{\mathcal{G}}^{x,n}_{i+1/2,j}\delta^n_{i+1/2,j}+ \tilde{\mathcal{G}}^{x,n}_{i-1/2,j}(1-\delta^n_{i-1/2,j})+\dfrac{\mathcal D^1_{i+1/2,j} + \mathcal D^1_{i-1/2,j}}{\Delta w} \right)  f_{i,j}^n \right.\\
&+ \left( - \tilde{\mathcal{G}}^{x,n}_{i-1/2,j}\delta^n_{i-1/2,j}+\dfrac{\mathcal D^1_{i-1/2,j}}{\Delta w}\right) f_{i-1,j}^n \Bigg]+\dfrac{\Delta t}{\Delta w}\left[ \left(  \tilde{\mathcal{G}}^{y,n}_{i,j+1/2}(1-\delta^n_{i,j+1/2})+\dfrac{\mathcal D^2_{i,j+1/2}}{\Delta w}\right) f_{i,j+1}^n  \right.  \\
&\left. - \left(- \tilde{\mathcal{G}}^{y,n}_{i,j+1/2}\delta^n_{i,j+1/2}+ \tilde{\mathcal{G}}^{y,n}_{i,j-1/2}(1-\delta^n_{i,j-1/2})+\dfrac{\mathcal D^2_{i,j+1/2} + \mathcal D^2_{i,j-1/2}}{\Delta w}\right) f_{i,j}^n \right. \\
&+\left( - \tilde{\mathcal{G}}^{y,n}_{i,j-1/2}\delta^n_{i,j-1/2}+\dfrac{\mathcal D^2_{i,j-1/2}}{\Delta w}\right) f_{i,j-1}^n \Bigg]. 
\end{split}
\]
This is a sum of convex combinations of $f_{i+1,j}^n$, $f_{i-1,j}^n$ and $f_{i,j+1}^n$,$f_{i,j-1}^n$ if the following conditions are satisfied 

\[
\begin{aligned}
& \tilde{\mathcal{G}}^{x,n}_{i+1/2,j}(1-\delta^n_{i+1/2,j})+\dfrac{\mathcal D^1_{i+1/2,j}}{\Delta w}\geq 0, \qquad
& - \tilde{\mathcal{G}}^{x,n}_{i-1/2,j}\delta^n_{i-1/2,j}+\dfrac{\mathcal D^1_{i-1/2,j}}{\Delta w} \geq 0, \\
& \tilde{\mathcal{G}}^{y,n}_{i,j+1/2}(1-\delta^n_{i,j+1/2})+\dfrac{\mathcal D^2_{i,j+1/2}}{\Delta w} \geq 0,
& - \tilde{\mathcal{G}}^{y,n}_{i,j-1/2}\delta^n_{i,j-1/2}+\dfrac{\mathcal D^2_{i,j-1/2}}{\Delta w}\geq 0 ,
\end{aligned}
\]
that is equivalent to
\[
\begin{aligned}
&  \lambda^n_{i+1/2,j}\left(1-\dfrac{1}{1-\exp (\lambda^n_{i+1/2,j})} \right) \geq 0, \qquad
& \dfrac{\lambda^n_{i-1/2,j}}{\exp (\lambda^n_{i-1/2,j})-1}  \geq 0, \\
& \lambda^n_{i,j+1/2}\left(1-\dfrac{1}{1-\exp (\lambda^n_{i,j+1/2})} \right) \geq 0, \qquad
& \dfrac{\lambda^n_{i,j-1/2}}{\exp (\lambda^n_{i,j-1/2})-1}  \geq 0,
\end{aligned}
\]
which hold true thanks to the basic properties of the exponential function. In order to ensure positivity for $f_{i,j}^{n+1}$ if $f_{i,j}^n \geq 0$ we must have for all $i,j$
\[
 \left( 1- (\nu_x+\nu_y) \dfrac{\Delta t}{\Delta w}\right) f_{i,j}^n \geq 0, 
\]
where 
\[
\begin{split}
\nu_x  &= \max_{i,j} \left\{-\tilde{\mathcal{G}}^{x,n}_{i+1/2,j}\delta^n_{i+1/2,j} + \tilde{\mathcal{G}}^{x,n}_{i-1/2,j}(1-\delta^n_{i-1/2,j})+\dfrac{\mathcal D^1_{i+1/2,j}+ \mathcal D^1_{i-1/2,j}}{\Delta w}\right\}, \\
\nu_y &= \max_{i,j} \left\{ -\tilde{\mathcal{G}}^{y,n}_{i,j+1/2}\delta^n_{i,j+1/2}+ \tilde{\mathcal{G}}^{y,n}_{i,j-1/2}(1-\delta^n_{i,j-1/2})+\dfrac{\mathcal D^2_{i,j+1/2} + \mathcal D^2_{i,j-1/2}}{\Delta w} \right\},
\end{split}
\]
from which we can conclude as $0\le\delta_{i\pm 1/2,j}\le 1$, $0\le \delta_{i,j\pm 1/2} \le 1$. 
\end{proof}

\textcolor{black}{We highlight how the CFL criterion in \eqref{pos.eu} ensures positivity of the numerical solution of the problem.} This remarkable property holds also for higher order strong stability preserving (SSP) methods like SSP Runge-Kutta and SSP multistep methods \cite{GST} since these are convex combinations of the forward Euler integration. Hence, the proved non-negativity of the scheme is automatically extended to each general SSP type time integration. \\

Even if in \eqref{pos.eu} we obtained an effective time step bound for the positivity of the explicit scheme, for practical purposes this parabolic restriction is very heavy especially in genuine nonlinear type problems. Usually the strategy to overcome this problem relies in the adoption of IMEX schemes \cite{DP}. Nevertheless, this is not always possible due to the strong nonlinearities embedded in problem \eqref{FP1} coming from the nonlocal drift term. \textcolor{black}{Furthermore}, the defined weights $\delta_{i+1/2,j},\delta_{i,j+1/2}$ depend in general on $f$ introducing additional difficulties. An efficient way to overcome this problem relies in the semi-implicit integration technique, see \cite{BFR}. 

\subsection{\textcolor{black}{Positivity  of the semi-implicit scheme}}
To apply semi-implicit schemes we integrate \eqref{eq:cons_disc} as follows
\begin{equation}\label{s.impl}
 f_{i,j}^{n+1} = f_{i,j}^n  + \Delta t \dfrac{\hat{\mathcal{F}}^{x,n+1}_{i+1/2,j}-\hat{\mathcal{F}}^{x,n+1}_{i-1/2,j}}{\Delta w}+\Delta t\dfrac{\hat{\mathcal{F}}^{y,n+1}_{i,j+1/2}-\hat{\mathcal{F}}^{y,n+1}_{i,j-1/2}}{\Delta w},
\end{equation}
where now the discretized flux terms $\hat {\mathcal F}^{x,n+1}_{i+1/2,j}$, $\hat{\mathcal F}^{y,n+1}_{i,j+ 1/2}$ are defined as
\[
\begin{split}
\hat{\mathcal{F}}^{x,n+1}_{i+1/2,j}=\tilde{\mathcal{G}}^{x,n}_{i+1/2,j}\left[(1-\delta^n_{i+1/2,j})f^{n+1}_{i+1,j}+ \delta^n_{i+1/2,j}f^{n+1}_{i,j}\right] +\mathcal D^1_{i+1/2,j}\dfrac{f^{n+1}_{i+1,j}-f^{n+1}_{i,j}}{\Delta w},\\[10pt]
\hat{\mathcal{F}}^{y,n+1}_{i,j+1/2}=\tilde{\mathcal{G}}^{y,n}_{i,j+1/2}\left[(1-\delta^n_{i,j+1/2})f^{n+1}_{i,j+1}+\delta^n_{i,j+1/2}f^{n+1}_{i,j}\right]+\mathcal D^2_{i,j+1/2}\dfrac{f^{n+1}_{i,j+1}-f^{n+1}_{i,j}}{\Delta w}.
\end{split}
\]

\begin{theorem}\label{prop3}
Under the time step restriction 
$$
\Delta t \leq \dfrac{\Delta w}{2(G_x + G_y)}, \qquad G_x = \max_{i,j,n}\{|\tilde{\mathcal{G}}^{x,n}_{i+1/2,j}| \}, \quad  G_y = \max_{i,j,n}\{|\tilde{\mathcal{G}}^{y,n}_{i,j+1/2}| \},
$$
the semi-implicit scheme (\ref{s.impl}) preserves nonnegativity, $\ie$, $f_{i,j}^{n+1}\geq 0$ if $f_{i,j}^n \geq 0$.
\end{theorem}
\begin{proof}
Equation (\ref{s.impl}) corresponds to
\[
\begin{split}
&f_{i+1,j}^{n+1}\left\{- \dfrac{\Delta t}{\Delta w} \left[  \tilde{\mathcal{G}}^{x,n}_{i+1/2,j}(1-\delta_{i+1/2,j})+\dfrac{\mathcal D^1_{i+1/2,j}}{\Delta w}\right]\right\} \\ 
&+f_{i,j}^{n+1}\left\{ 1-\dfrac{\Delta t}{\Delta w}\left[  \tilde{\mathcal{G}}^{x,n}_{i+1/2,j}\delta^n_{i+1/2,j}- \tilde{\mathcal{G}}^{x,n}_{i-1/2,j}(1-\delta^n_{i-1/2,j})  - \dfrac{\mathcal D^1_{i+1/2,j} + \mathcal D^1_{i-1/2,j}}{\Delta w}  \right]\right\}\\
&+f_{i-1,j}^{n+1}\left\{ -\dfrac{\Delta t}{\Delta w}\left[ - \tilde{\mathcal{G}}^{x,n}_{i-1/2,j}\delta^n_{i-1/2,j}+\dfrac{\mathcal D^1_{i-1/2,j}}{\Delta w}\right]\right\} \\
& + f_{i,j+1}^{n+1}\left\{-\dfrac{\Delta t}{\Delta w}\left[  \tilde{\mathcal{G}}^{y,n}_{i,j+1/2}(1-\delta^n_{i,j+1/2})+\dfrac{\mathcal D^2_{i,j+1/2}}{\Delta w}\right]\right\} \\
&+f_{i,j}^{n+1}\left\{ 1-\dfrac{\Delta t}{\Delta w}\left[ \tilde{\mathcal{G}}^{y,n}_{i,j+1/2}\delta^n_{i,j+1/2}- \tilde{\mathcal{G}}^{y,n}_{i,j-1/2}(1-\delta^n_{i,j-1/2})-\dfrac{\mathcal D^2_{i,j+1/2}+\mathcal D^2_{i,j-1/2}}{\Delta w} \right]\right\}\\
&+f_{i,j-1}^{n+1}\left\{ \left[ - \dfrac{\Delta t}{\Delta w}\tilde{\mathcal{G}}^{y,n}_{i,j-1/2}\delta^n_{i,j-1/2}+\dfrac{\mathcal D^2_{i,j-1/2}}{\Delta w}\right] \right\} = f_{i,j}^{n}.
\end{split}
\]
Using the identities in \eqref{lambda}, we have that
\[
\begin{split}
&f_{i+1,j}^{n+1}\left\{-\dfrac{\Delta t}{\Delta w^{2}}  \mathcal D^1_{i+1/2,j} \dfrac{\lambda^n_{i+1/2,j}}{\exp (\lambda^n_{i+1/2,j})-1}\exp(\lambda^n_{i+1/2,j}) \right\} \\
&+f_{i,j}^{n+1}\left\{ 1+\dfrac{\Delta t}{\Delta w^{2}}\left[  \mathcal D^1_{i+1/2,j}\dfrac{\lambda^n_{i+1/2,j}}{\exp (\lambda^n_{i+1/2,j})-1}  + \mathcal D^1_{i-1/2,j} \dfrac{\lambda^n_{i-1/2,j}}{\exp (\lambda^n_{i-1/2,j})-1}\exp(\lambda^n_{i-1/2,j}) \right]\right\} \\
&+f_{i-1,j}^{n+1}\left\{ -\dfrac{\Delta t}{\Delta w^{2}} \mathcal D^1_{i-1/2,j} \dfrac{\lambda^n_{i-1/2,j}}{\exp (\lambda^n_{i-1/2,j})-1}\right\} \\
&+f_{i,j+1}^{n+1}\left\{-\dfrac{\Delta t}{\Delta w^{2}} \mathcal D^2_{i,j+1/2} \dfrac{\lambda^n_{i,j+1/2}}{\exp (\lambda^n_{i,j+1/2})-1}\exp(\lambda^n_{i,j+1/2})\right\}\\
&+f_{i,j}^{n+1}\left\{ 1+\dfrac{\Delta t}{\Delta w^{2}}\left[ \mathcal D^2_{i,j+1/2} \dfrac{\lambda^n_{i,j+1/2}}{\exp (\lambda^n_{i,j+1/2})-1}  + \mathcal D^2_{i,j-1/2}\dfrac{\lambda^n_{i,j-1/2}}{\exp (\lambda^n_{i,j-1/2})-1}\exp(\lambda^n_{i,j-1/2})   \right]\right\} \\
&+f_{i,j-1}^{n+1}\left\{  - \dfrac{\Delta t}{\Delta w^{2}} \mathcal D^2_{i,j-1/2} \dfrac{\lambda^n_{i,j-1/2}}{\exp (\lambda^n_{i,j-1/2})-1} \right\} = f_{i,j}^{n}. 
\end{split}
\]
Then by introducing the quantities 
\[
\alpha^n_{i+1/2,j}=\dfrac{\lambda^n_{i+1/2,j}}{\exp (\lambda^n_{i+1/2,j})-1} \geq 0 \qquad \textrm{and} \qquad\alpha^n_{i,j+1/2}=\dfrac{\lambda^n_{i,j+1/2}}{\exp (\lambda^n_{i,j+1/2})-1} \geq 0
\]
and setting
\[
\begin{array}{lr}
R_x(j)_{i}^n=1+\dfrac{\Delta t}{\Delta w^{2}}\left[  \mathcal D^1_{i+1/2,j}\alpha^n_{i+1/2,j}  - \mathcal D^1_{i-1/2,j}\alpha^n_{i-1/2,j}\exp(\lambda^n_{i-1/2,j})  \right],\\[10pt]
Q_x(j)_{i}^n=-\dfrac{\Delta t}{\Delta w^{2}} \mathcal D^1_{i+1/2,j}\alpha^n_{i+1/2,j}\exp(\lambda^n_{i+1/2,j}),\\[10pt]
P_x(j)_{i}^n=-\dfrac{\Delta t}{\Delta w^{2}} \mathcal D^1_{i-1/2,j}\alpha^n_{i-1/2,j},\\[10pt]
R_y(i)_{j}^n=1+\dfrac{\Delta t}{\Delta w^{2}}\Big[ \mathcal D^2_{i,j+1/2}\alpha^n_{i,j+1/2} - \mathcal D^2_{i,j-1/2}\alpha^n_{i,j-1/2}\exp(\lambda^n_{i,j-1/2})  \Big],\\[10pt]
Q_y(i)_{j}^n=-\dfrac{\Delta t}{\Delta w^{2}}\mathcal D^2_{i,j+1/2}\alpha^n_{i,j+1/2}\exp(\lambda^n_{i,j+1/2}),\\[10pt]
P_y(i)_{j}^n=- \dfrac{\Delta t}{\Delta w^{2}}\mathcal D^2_{i,j-1/2}\alpha^n_{i,j-1/2},
\end{array}
\]
the latter equation reduces to
\[
\begin{split}
&R_x(j)_{i}^n f_{i,j}^{n+1}-Q_x(j)_{i}^n f_{i+1,j}^{n+1}-P_x(j)_{i}^n f_{i-1,j}^{n+1}\\
&\qquad\qquad+R_y(i)_{j}^n f_{i,j}^{n+1}-Q_y(i)_{j}^n f_{i,j+1}^{n+1}-P_y(i)_{j}^n f_{i,j-1}^{n+1} = f_{i,j}^n.
\end{split}
\]
Now, by denoting $\textbf{f}^n = \left\{f_{i,j}^n \right\}_{i=1,\dots,N}^{j=1,\dots,N} $ we can define the matrices
\[
\begin{split}
\mathcal{A}_x[{\textbf{f}^n}]_{ik}&=
\begin{cases}
R_x(j)_{i}^n   & k=i,\\
-Q_x(j)_{i}^n  & k=i+1, \quad 0\le i\le N-1, \\
-P_x(j)_{i}^n      & k=i-1, \quad 1\le i \le N ,
\end{cases}
\\
\mathcal{A}_y[{\textbf{f}^n}]_{jk}&=
\begin{cases}
R_y(i)_{j}^n     & k=j,   \\
-Q_y(i)_{j}^n  & k=j+1, \quad  0\le j\le N-1,  \\
-P_y(i)_{j}^n  & k=j-1, \quad 1\le j \le N, 
\end{cases}
\end{split}
\]
and we reduce to study
\[
\left(\mathcal{A}_x[\textbf{f}^n] + \mathcal{A}_y[\textbf{f}^n]\right) \textbf{f}^{n+1} = \textbf{f}^n. 
\]
 If $\textbf{f}^n \geq 0$, in order to prove that $\textbf{f}^{n+1}\geq 0$ it is sufficient to prove that $\left(\mathcal{A}_x[\textbf{f}^n] + \mathcal{A}_y[\textbf{f}^n]\right)^{-1}$ is nonnegative. Let us observe that since $\left(\mathcal{A}_x[\textbf{f}^n] + \mathcal{A}_y[\textbf{f}^n]\right)$ is tridiagonal we only need to prove that it is a diagonally dominant matrix. In particular, this is true if for each $i,j=1,\dots,N$ the following inequality is verified
 \[
|R_x(j)_{i}^n+R_y(i)_{j}^n| > |Q_x(j)_{i}^n+Q_y(i)_{j}^n|+|P_x(j)_{i}^n+P_y(i)_{j}^n|, 
\]
which is true provided
\[
\begin{split}
1>&\dfrac{\Delta t}{ \Delta w^{2}}\left[\mathcal D^1_{i+1/2,j}\alpha^n_{i+1/2,j}( \exp (\lambda^n_{i+1/2,j})-1)-\mathcal D^1_{i-1/2,j}\alpha^n_{i-1/2,j}( \exp (\lambda^n_{i-1/2,j})-1) \right]\\
&+\dfrac{\Delta t}{ \Delta w^{2}}\Big[\mathcal D^2_{i,j+1/2}\alpha^n_{i,j+1/2}( \exp (\lambda^n_{i,j+1/2})-1)-\mathcal D^2_{i,j-1/2}\alpha^n_{i,j-1/2}( \exp (\lambda^n_{i,j-1/2})-1) \Big] \\
=& \dfrac{\Delta t}{\Delta w^2} \left[ \mathcal D_{i+1/2,j}^1 \lambda_{i+1/2,j}^n - \mathcal D_{i-1/2,j}^1 \lambda_{i-1/2,j}^n + \mathcal D_{i,j-1/2}^2 \lambda_{i,j+1/2}^n  - \mathcal D_{i,j-1/2}^2\lambda_{i,j-1/2}^n  \right] \\
=& \dfrac{\Delta t}{\Delta w} \left[ \tilde{\mathcal{G}}^{x,n}_{i+1/2,j}- \tilde{\mathcal{G}}^{x,n}_{i-1/2,j} + \tilde{\mathcal{G}}^{y,n}_{i,j+1/2}- \tilde{\mathcal{G}}^{y,n}_{i,j-1/2}\right]. 
\end{split}
\]
\end{proof}

\begin{remark}
 Fully-implicit schemes require a special treatment since the nonlinearity in the drift term poses nontrivial questions at the numerical level. A possible way to overcome this difficulty is to use iterative methods as suggested in \cite{Pareschi_Zanella}. This issue anyway goes beyond  the goals of the present manuscript and we postpone this discussion to future works. 
\end{remark}

\section{Trend to equilibrium} \label{sect:entropy}
A classical question in kinetic theory pertains to the determination of the rate of exponential convergence to equilibrium. To this end the consolidated approach relies on entropy production arguments for which lower bounds are explicitly computable thanks to log-Sobolev inequalities, see \cite{Tosc0,TV}. In particular, the convergence to the stationary state of the standard Fokker-Planck equation can be achieved by looking at the monotonicity in time of various Lyapunov functionals like the relative entropy. In the nonconstant diffusion case additional difficulties arise since standard log-Sobolev inequality are not available \cite{MJT}.

\subsection{\textcolor{black}{Steady state and vanishing flux for linear problems}}
\textcolor{black}{
In order to study the entropy properties, as done typically \cite{Tosc0,TV,Furioli}, we consider the linear problem defined by $\mathcal{B}[f](w,t)=B(w)$.
 Moreover, we suppose that a stationary state exists and that, coherently with the assumptions already discussed, the flux vanishes at the stationary state, i.e.  $\mathcal{F}^{\infty}\textcolor{black}{(w)}=0$. }
The latter is equivalent to say that $f^{\infty}\textcolor{black}{(w)}$  satisfies
\[
B(w)f^{\infty}\textcolor{black}{(w)}+\nabla_w \cdot \big( \mathbb{D}\textcolor{black}{(w)}f^{\infty}\textcolor{black}{(w)}\big)=0, \qquad w \in  \Omega.
\] 
Then we have
\begin{equation}\label{w-U}
B(w)=-\dfrac{f^{\infty}\textcolor{black}{(w)}\nabla_w \cdot \mathbb{D}\textcolor{black}{(w)}}{f^{\infty}\textcolor{black}{(w)}}-\mathbb{D}\textcolor{black}{(w)}\dfrac{\nabla_w f^{\infty}\textcolor{black}{(w)}}{f^{\infty}\textcolor{black}{(w)}}=-\nabla_w \cdot \mathbb{D}\textcolor{black}{(w)}-\mathbb{D}\textcolor{black}{(w)}\dfrac{\nabla_w f^{\infty}\textcolor{black}{(w)}}{f^{\infty}\textcolor{black}{(w)}},
\end{equation}
see \cite{Risken}. Hence, we can rewrite our problem in the form
\begin{equation}\label{landau}
\partial_t f\textcolor{black}{(w,t)}= \nabla_{w}\cdot \left[ f^{\infty}\textcolor{black}{(w)} \mathbb{D}\textcolor{black}{(w)} \nabla_w \frac{f\textcolor{black}{(w,t)}}{f^{\infty}\textcolor{black}{(w)}}\right], 
\end{equation}
since
\[
\begin{split}
&\nabla_{w}\cdot \left[ B(w)f\textcolor{black}{(w,t)}+\nabla_w \cdot \big( \mathbb{D}\textcolor{black}{(w)}f\textcolor{black}{(w,t)}\big)\right] \\
&\qquad=\nabla_w \cdot \left[-f\textcolor{black}{(w,t)}\nabla_w \cdot \mathbb{D}\textcolor{black}{(w)}-f\textcolor{black}{(w,t)}\mathbb{D}\textcolor{black}{(w)} \dfrac{\nabla_w f^{\infty}\textcolor{black}{(w)}}{f^{\infty}\textcolor{black}{(w)}}+\nabla_w \cdot \big( \mathbb{D}\textcolor{black}{(w)}f\textcolor{black}{(w,t)}\big)\right]\\
&\qquad=\nabla_w  \cdot\left[ -f\textcolor{black}{(w,t)}\mathbb{D}\textcolor{black}{(w)} \dfrac{\nabla_w f^{\infty}\textcolor{black}{(w)}}{f^{\infty}\textcolor{black}{(w)}}+\mathbb{D}\textcolor{black}{(w)}\nabla_w  f\textcolor{black}{(w,t)}\right] \\
&\qquad=\nabla_w  \cdot\left[f\textcolor{black}{(w,t)}\mathbb{D}\textcolor{black}{(w)} \left( \dfrac{\nabla_w f\textcolor{black}{(w,t)}}{f\textcolor{black}{(w,t)}}-\dfrac{\nabla_w f^{\infty}\textcolor{black}{(w)}}{f^{\infty}\textcolor{black}{(w)}}\right) \right] \\
&\qquad =\nabla_w  \cdot\left[f\textcolor{black}{(w,t)}\mathbb{D}\textcolor{black}{(w)}\nabla_w \log \left(\dfrac{f\textcolor{black}{(w,t)}}{f^{\infty}\textcolor{black}{(w)}}\right) \right]\\
&\qquad  = \nabla_{w}\cdot \left[ f^{\infty}\textcolor{black}{(w)} \mathbb{D}\textcolor{black}{(w)} \nabla_w \frac{f\textcolor{black}{(w,t)}}{f^{\infty}\textcolor{black}{(w)}}\right].
\end{split}
\]
\textcolor{black}{The no-flux boundary conditions in this case read
$$
\left[ f^{\infty}\textcolor{black}{(w)} \mathbb{D}\textcolor{black}{(w)}  \nabla_w \frac{f\textcolor{black}{(w,t)}}{f^{\infty}\textcolor{black}{(w)}}\right]\cdot \boldsymbol{n}(w)=0, \qquad w\in \partial \Omega.
$$
}
Therefore, from the Landau's formulation \eqref{landau}, we get the equation satisfied by $F\textcolor{black}{(w,t)}=f\textcolor{black}{(w,t)}/f^{\infty}\textcolor{black}{(w)}$ that is
\[
\begin{split}
\partial_t F\textcolor{black}{(w,t)}  =\dfrac{\partial_t f\textcolor{black}{(w,t)}}{f^{\infty}\textcolor{black}{(w)}} &=\dfrac{\nabla_{w}\cdot \left[ f^{\infty}\textcolor{black}{(w)} \mathbb{D}\textcolor{black}{(w)}  \nabla_w F\textcolor{black}{(w,t)}\right]}{f^{\infty}\textcolor{black}{(w)}} \\
& = \nabla_w\cdot \big( \mathbb{D}\textcolor{black}{(w)}\nabla_w F\textcolor{black}{(w,t)}\big)+\big(\mathbb{D}\textcolor{black}{(w)} \nabla_w F\big)\cdot \dfrac{\nabla_w f^{\infty}\textcolor{black}{(w)}}{f^{\infty}\textcolor{black}{(w)}}\\
&=\nabla_w\cdot \big( \mathbb{D}\textcolor{black}{(w)} \nabla_w F\textcolor{black}{(w,t)}\big)-B(w) \cdot \nabla_w F\textcolor{black}{(w,t)}-(\nabla_w \cdot \mathbb{D}\textcolor{black}{(w)})\cdot \nabla_w F\textcolor{black}{(w,t)},
\end{split}
\]
where the last equality holds true since $\mathbb{D}\textcolor{black}{(w)}$ is a symmetric matrix $\forall w\in \Omega$ and thanks to the relation \eqref{w-U}. Now, since
\begin{equation}\label{prod.ten}
\nabla_w\cdot \Big( \mathbb{D}\textcolor{black}{(w)}\nabla_w F\textcolor{black}{(w,t)}\Big)=(\nabla_w \cdot \mathbb{D}\textcolor{black}{(w)}) \cdot \nabla_w  F\textcolor{black}{(w,t)}+\mathbb{D}\textcolor{black}{(w)}:\nabla_w (\nabla_wF\textcolor{black}{(w,t)}),
\end{equation}
where $\nabla_w (\nabla_w F\textcolor{black}{(w,t)})$ is the covariant derivative of the vector $\nabla_w F\textcolor{black}{(w,t)}$, \ie  $\;\nabla_w (\nabla_w F\textcolor{black}{(w,t)})=(\partial _{w_i} \nabla_w F\textcolor{black}{(w,t)})=(\partial _{w_i} \partial_{w_j}F\textcolor{black}{(w,t)})$, and it is the Hessian matrix of $F$, which we will denote $H_w[F]$. \textcolor{black}{With} $:$ we indicated the inner tensorial product that is for definition
\[
\mathbb{D}\textcolor{black}{(w)} : H_w[F]\textcolor{black}{(w,t)}={\rm tr}\Big[ \big(H_w[F]\textcolor{black}{(w,t)}\big)^T \mathbb{D}\textcolor{black}{(w)}\Big].
\]
In conclusion, we obtain
\begin{equation}\label{eq.F}
\partial_t F\textcolor{black}{(w,t)}=\mathbb{D}\textcolor{black}{(w)} : H_w[F]\textcolor{black}{(w,t)}- B(w) \cdot \nabla_w F\textcolor{black}{(w,t)}.
\end{equation}
\subsection{Lyapunov functionals}
We will focus on the study of relative Shannon entropy for the problem \eqref{FP1} with  nonconstant diffusion.  We will extend the results proved in \cite{Furioli} to the two-dimensional case where the diffusion is a nonconstant positive definite tensor of the second order and the drift term is general in the form $B(w)$.

Let $f,g : \Omega \longmapsto \mathbb{R}^+$ denote two probability densities.
Then, the relative Shannon entropy of $f$ and $g$ is defined by 
\begin{equation}\label{Shannon}
H(f|g)=\int_{\Omega} f\log \dfrac{f}{g} dw.
\end{equation}
It is a Lyapunov functional since the following result can be established. 
\begin{theorem}\label{teo.1}
Let us suppose thet $\mathcal{F}^{\infty}\textcolor{black}{(w)}=0$ and that the drift is of the form \eqref{w-U}. Let $F(w,t)$ be the solution of \eqref{eq.F} in $\Omega$. Then, if $\Psi(w)$ is a smooth function such that
$$
|\Psi| \leq c \le \infty \quad \textit{on} \quad \partial \Omega,
$$
then the following relation holds
$$
\int_{\Omega}f^{\infty}(w,t)\Psi(w) \partial_t F(w,t) dw=\int_{\Omega}f^{\infty}(w,t)\nabla_w\Psi\cdot \left(\mathbb{D}\textcolor{black}{(w)} \nabla_w F(w,t) \right)dw. 
$$
\end{theorem}
\begin{proof}
From \eqref{eq.F} it follows that
\[
\begin{split}
&\displaystyle \int_{\Omega} f^{\infty}(w)\Psi(w) \partial_t F\textcolor{black}{(w,t)} dw =
\displaystyle \int_{\Omega}f^{\infty}(w)\Psi(w)  \big(\mathbb{D}\textcolor{black}{(w)} : H_w[F]\textcolor{black}{(w,t)}- B(w) \cdot \nabla_w F\textcolor{black}{(w,t)} \big) dw\\
\end{split}
\]
and from \eqref{Shannon} the latter term is equal to 
\[
\begin{split}
& \displaystyle \int_{\Omega}f^{\infty}(w)\Psi(w) \Big[ \nabla_w\big(\mathbb{D}\textcolor{black}{(w)} \nabla_w F\textcolor{black}{(w,t)}\Big)-\nabla_w \cdot \mathbb{D}\textcolor{black}{(w)} \nabla_w  F\textcolor{black}{(w,t)} \Big]dw- \displaystyle \int_{\Omega}f^{\infty}(w)\Psi(w)B(w) \cdot \nabla_w F\textcolor{black}{(w,t)}  dw\\
&\quad =-\displaystyle \int_{\Omega}\nabla_w \big(f^{\infty}(w)\Psi(w)\big) \cdot \left(\mathbb{D}\textcolor{black}{(w)} \nabla_w F\textcolor{black}{(w,t)}\right) dw +\displaystyle \oint_{\partial \Omega}\Psi(w)f^{\infty}(w)(\mathbb{D}\textcolor{black}{(w)}\nabla_wF\textcolor{black}{(w,t)} )\cdot \boldsymbol{n}(w) d\sigma(w) \\
&\qquad-\displaystyle \int_{\Omega}\Big[B(w)f^{\infty}(w)+\nabla_w\cdot\mathbb{D}\textcolor{black}{(w)}f^{\infty}(w)\Big] \cdot \nabla_w F\textcolor{black}{(w,t)} \Psi(w)  dw\\
&\quad=-\displaystyle \int_{\Omega}\nabla_w \big(f^{\infty}(w)\Psi(w) \big)\cdot \big( \mathbb{D}\textcolor{black}{(w)} \nabla_wF\textcolor{black}{(w,t)}\big)dw \\
&\qquad -\displaystyle \int_{\Omega}\Big[B(w)f^{\infty}(w)+\nabla_w\cdot\mathbb{D}\textcolor{black}{(w)}f^{\infty}(w)\Big] \cdot \nabla_w F\textcolor{black}{(w,t)} \Psi(w)  dw\\
&\quad = -\displaystyle \int_{\Omega}\Psi(w)\nabla_w f^{\infty}(w) \cdot \big( \mathbb{D}\textcolor{black}{(w)} \nabla_wF\textcolor{black}{(w,t)}\big)dw-\displaystyle \int_{\Omega}f^{\infty}(w)\nabla_w\Psi(w)  \cdot \big( \mathbb{D}\textcolor{black}{(w)} \nabla_wF\textcolor{black}{(w,t)}\big)dw\\
&\qquad -\displaystyle \int_{\Omega}\Big[B(w)f^{\infty}(w)+\nabla_w\cdot\mathbb{D}\textcolor{black}{(w)}f^{\infty}(w)\Big] \cdot \nabla_w F\textcolor{black}{(w,t)} \Psi(w)  dw\\
&\quad =-\displaystyle \int_{\Omega}f^{\infty}(w)\nabla_w\Psi(w)  \cdot \big( \mathbb{D}\textcolor{black}{(w)} \nabla_wF\textcolor{black}{(w,t)}\big)dw\\
& \quad\quad-\displaystyle \int_{\Omega}\Big[B(w)f^{\infty}(w)+\nabla_w\cdot \big(\mathbb{D}\textcolor{black}{(w,t)}f^{\infty}(w)\big)\Big] \cdot \nabla_w F\textcolor{black}{(w,t)} \Psi(w)  dw\\
&\quad=-\displaystyle \int_{\Omega}f^{\infty}(w)\nabla_w\Psi(w)  \cdot \big( \mathbb{D}\textcolor{black}{(w,t)} \nabla_wF\textcolor{black}{(w,t)}\big)dw,
\end{split}
\]
as the border terms vanish because of the boundary conditions and where we used (\ref{prod.ten}), the divergence theorem and the fact that $\mathcal{F}^{\infty}\textcolor{black}{(w)}=0$.
\end{proof}
\begin{theorem}
Let us suppose that $\mathcal{F}^{\infty}\textcolor{black}{(w)}=0$ and that the drift is of the form \eqref{w-U}.
Let the smooth function $\Phi(x)$, $x \in \mathbb{R}^+$ be convex. Then, if $F(w,t)$
is the solution of \eqref{eq.F} in $\Omega$, and $c \leq F(w,t)\leq C$ for some
positive constants $c < C$, the functional
$$
\textcolor{black}{\Theta[F](t)}=\int_{\Omega}f^{\infty}(w) \Phi(F(w,t)) dw
$$
is monotonically decreasing in time, and the following equality holds
$$
\frac{d}{dt}\textcolor{black}{\Theta[F](t)}=-I_{\Theta}\textcolor{black}{[F](t)},
$$
where $I_{\Theta}\textcolor{black}{[F](t)}$ denotes the quantity
\begin{equation}\label{entropy.dec}
I_{\Theta}\textcolor{black}{[F](t)}=\int_{\Omega}f^{\infty}(w) \Phi''(F(w,t))\nabla_w F\textcolor{black}{(w,t)} \mathbb{D}(w) \nabla_w F\textcolor{black}{(w,t)} dw,
\end{equation}
that is non-negative because $\Phi$ is convex and $\mathbb{D}(w)$ is positive definite.
\end{theorem}
\begin{proof}
The relation (\ref{entropy.dec}) follows from Theorem \ref{teo.1} by choosing $\Psi(w)=\Phi'(F(w,t))$ for a fixed $t>0$.
\end{proof}
The Shannon entropy of $f\textcolor{black}{(w,t)}$ relative to $f^{\infty}\textcolor{black}{(w)}$, defined
by (\ref{Shannon}) with $g = f^{\infty}$, is obtained by choosing $\Phi(x) = x \log x$. In this case
$$
I_{\Theta}\textcolor{black}{[F](t)}=\int_{\Omega}f\textcolor{black}{(w,t)} \frac{\nabla_w F\textcolor{black}{(w,t)}}{F\textcolor{black}{(w,t)}} \mathbb{D}(w)\frac{\nabla_w F\textcolor{black}{(w,t)}}{F\textcolor{black}{(w,t)}} dw,
$$
that may be re-written as
$$
I_{\Theta}\textcolor{black}{[F](t)}=\int_{\Omega}f\textcolor{black}{(w,t)} \left(\frac{\nabla_w f\textcolor{black}{(w,t)}}{f\textcolor{black}{(w,t)}}-\frac{\nabla_w f^{\infty}\textcolor{black}{(w)}}{f^{\infty}\textcolor{black}{(w)}}\right) \mathbb{D}(w) \left(\frac{\nabla_w f\textcolor{black}{(w,t)}}{f\textcolor{black}{(w,t)}}-\frac{\nabla_w f^{\infty}\textcolor{black}{(w)}}{f^{\infty}\textcolor{black}{(w)}}\right) dw,
$$
that is the Fisher information of $f\textcolor{black}{(w,t)}$ relative to $f^{\infty}\textcolor{black}{(w)}$.
We might also consider the weighted $L^2$ distance that is obtained by considering $\Phi(x)=(x-1)^2$.
In this case
$$
\textcolor{black}{\Theta[F](t)}=L^2(f,f^{\infty})=\int_{\Omega}\dfrac{(f\textcolor{black}{(w,t)}-f^{\infty}\textcolor{black}{(w)})^2}{f^{\infty}\textcolor{black}{(w)}} dw
$$
and 
$$
I(\Theta)\textcolor{black}{[F](t)}=2\int_{\Omega}\nabla_wF\textcolor{black}{(w,t)} \mathbb{D}(w) \nabla_w F\textcolor{black}{(w,t)} dw. 
$$

\subsection{Dissipation of the numerical entropy}

In the following results we show how the derived schemes dissipate in the introduced setting a Shannon-type numerical entropy functional.

\begin{theorem}\label{prop4}
Let us consider a drift term of the form \eqref{w-U}. The numerical flux function \eqref{eq:flux_th} with $\delta_{i+1/2,j}, \delta_{i,j+1/2}$ given by  \eqref{delta} can be written in the form \eqref{landau} and reads
$$
\begin{cases}
\mathcal{F}^{x}_{i+1/2,j}\textcolor{black}{(t)}=\dfrac{\mathcal D^1_{i+1/2,j}}{\Delta w}\hat{f}^{\infty}_{i+1/2,j}\left( \dfrac{f_{i+1,j}\textcolor{black}{(t)}}{f^{\infty}_{i+1,j}} -\dfrac{f_{i,j}\textcolor{black}{(t)}}{f^{\infty}_{i,j}}\right),\\[10pt]
\mathcal{F}^{y}_{i,j+1/2}\textcolor{black}{(t)}=\dfrac{\mathcal D^2_{i,j+1/2}}{\Delta w}\hat{f}^{\infty}_{i,j+1/2}\left( \dfrac{f_{i,j+1}\textcolor{black}{(t)}}{f^{\infty}_{i,j+1}} -\dfrac{f_{i,j}\textcolor{black}{(t)}}{f^{\infty}_{i,j}}\right),
\end{cases}
$$
where
$$
\hat{f}^{\infty}_{i+1/2,j}=\dfrac{f^{\infty}_{i+1,j}f^{\infty}_{i,j}}{f^{\infty}_{i+1,j}-f^{\infty}_{i,j}}\log \Big( \dfrac{f^{\infty}_{i+1,j}}{f^{\infty}_{i,j}}\Big),\qquad
\hat{f}^{\infty}_{i,j+1/2}=\dfrac{f^{\infty}_{i,j+1}f^{\infty}_{i,j}}{f^{\infty}_{i,j+1}-f^{\infty}_{i,j}}\log \Big( \dfrac{f^{\infty}_{i,j+1}}{f^{\infty}_{i,j}}\Big).
$$
\end{theorem}

\begin{proof}
If $\mathcal{B}=B(w)$, we have that the definitions of $\lambda_{i+1/2,j}$ and $\lambda_{i,j+1/2}$ do not depend on time. Hence, we may denote $\lambda_{i+1/2,j}=\lambda^{\infty}_{i+1/2,j}$ and $\lambda_{i,j+1/2}=\lambda^{\infty}_{i,j+1/2}$ and we have
$$
\begin{array}{lcrl}
\log f^{\infty}_{i+1,j}-\log f^{\infty}_{i,j}=\lambda_{i+1/2,j},\\[10pt]
\log f^{\infty}_{i,j+1}-\log f^{\infty}_{i,j}=\lambda_{i,j+1/2},
\end{array}
$$
and $\delta_{i+1/2,j}, \delta_{i,j+1/2}$ are of the form \eqref{delta.inf}.
Therefore, under these assumptions the flux function writes
\begin{equation}\label{flux.inf}
\begin{split}
\mathcal{F}^{x}_{i+1/2,j}\textcolor{black}{(t)} & =\dfrac{\mathcal D^1_{i+1/2,j}}{\Delta w}\Big( \lambda_{i+1/2,j}\tilde{f}_{i+1/2,j}\textcolor{black}{(t)}+(f_{i+1,j}\textcolor{black}{(t)}-f_{i,j}\textcolor{black}{(t)})\Big)\\
& =\dfrac{\mathcal D^1_{i+1/2,j}}{\Delta w}\Big( \lambda_{i+1/2,j}\big(f_{i+1,j}\textcolor{black}{(t)}+\delta_{i+1/2,j}\textcolor{black}{(t)}(f_{i,j}\textcolor{black}{(t)}-f_{i+1,j}\textcolor{black}{(t)} )\big)+\big(f_{i+1,j}\textcolor{black}{(t)}-f_{i,j}\textcolor{black}{(t)}\big)\Big)
\end{split}
\end{equation}
and
\begin{equation}\label{flux.inf2}
\begin{split}
\mathcal{F}^{y}_{i,j+1/2}\textcolor{black}{(t)}&=\dfrac{\mathcal D^2_{i,j+1/2}}{\Delta w}\Big( \lambda_{i,j+1/2}\tilde{f}_{i,j+1/2}\textcolor{black}{(t)}+\big(f_{i,j+1}\textcolor{black}{(t)}-f_{i,j}\textcolor{black}{(t)}\big)\Big)\\
&=\dfrac{\mathcal D^2_{i,j+1/2}}{\Delta w}\Big( \lambda_{i,j+1/2}\big(f_{i,j+1}\textcolor{black}{(t)}+\delta_{i,j+1/2}(f_{i,j}\textcolor{black}{(t)}-f_{i,j+1}\textcolor{black}{(t)} )\big)+\big(f_{i,j+1}\textcolor{black}{(t)}-f_{i,j}\textcolor{black}{(t)}\big)\Big).
\end{split}
\end{equation}
By substituting \eqref{delta.inf} in \eqref{flux.inf}-\eqref{flux.inf2} we obtain the thesis.
\end{proof}

\begin{theorem}\label{th:entropy_diss}
Let us consider a drift term of the form \eqref{w-U}. The numerical
flux  \eqref{eq:flux_th} satisfies the discrete entropy
dissipation
$$
\dfrac{d}{dt} \mathcal{H}_{\Delta } (f,f^{\infty})\textcolor{black}{(t)}=-\mathcal{I}_{\Delta}(f,f^{\infty})\textcolor{black}{(t)},
$$
where
$$
\mathcal{H}_{\Delta } (f,f^{\infty})\textcolor{black}{(t)}=\Delta w^2 \sum_{j=0}^{N}\sum_{i=0}^{N} f_{i,j}\textcolor{black}{(t)} \log \dfrac{f_{i,j}\textcolor{black}{(t)}}{f^{\infty}_{i,j}}
$$
and $\mathcal{I}_{\Delta}$ is the positive discrete dissipation function
\begin{equation}\label{dissipation}
\begin{split}
\mathcal{I}_{\Delta }\textcolor{black}{(t)}=&\Delta w\sum_{j=0}^{N}\sum_{i=0}^{N}  \left[\log \left( \dfrac{f_{i+1,j}\textcolor{black}{(t)}}{f^{\infty}_{i+1,j}}\right)-\log \left(\dfrac{f_{i,j}\textcolor{black}{(t)}}{f^{\infty}_{i,j}}\right)\right]\left( \dfrac{f_{i+1,j}\textcolor{black}{(t)}}{f^{\infty}_{i+1,j}} -\dfrac{f_{i,j}\textcolor{black}{(t)}}{f^{\infty}_{i,j}}\right)\hat{f}^{\infty}_{i+1/2,j} \mathcal D^1_{i+1/2,j}\\
&+\sum_{i=0}^{N}\sum_{j=0}^{N} f_{i,j+1}\textcolor{black}{(t)} \left[ \log \left( \dfrac{f_{i,j+1}\textcolor{black}{(t)}}{f^{\infty}_{i,j+1}}\right)-\log \left(\dfrac{f_{i,j}\textcolor{black}{(t)}}{f^{\infty}_{i,j}}\right)\right]\left( \dfrac{f_{i,j+1}\textcolor{black}{(t)}}{f^{\infty}_{i,j+1}} -\dfrac{f_{i,j}\textcolor{black}{(t)}}{f^{\infty}_{i,j}}\right)\hat{f}^{\infty}_{i,j+1/2}\mathcal D^2_{i,j+1/2}.
\end{split}
\end{equation}
\end{theorem}

\begin{proof}
If we compute the time derivative of the discrete relative entropy we have that
\begin{equation*}\begin{split}
\dfrac{d}{dt} \mathcal{H}_{\Delta } (f,f^{\infty})\textcolor{black}{(t)} &=\Delta w^2\sum_{j=0}^{N}\sum_{i=0}^{N}  \dfrac{d f_{i,j}\textcolor{black}{(t)}}{dt}\left(1+\log\left(\dfrac{f_{i,j}\textcolor{black}{(t)}}{f_{i,j}^{\infty}}\right)\right)\\
&=\Delta w\sum_{j=0}^{N}\sum_{i=0}^{N}  \left(1+\log\left(\dfrac{f_{i,j}\textcolor{black}{(t)}}{f_{i,j}^{\infty}}\right)\right) \\
&\qquad \times \left( \mathcal{F}^{x}_{i+1/2,j}(t)-\mathcal{F}^x_{i-1/2,j}(t)+\mathcal{F}^y_{i,j+1/2}(t)-\mathcal{F}^y_{i,j-1/2}(t)\right).
\end{split}\end{equation*}
After telescopic summation and thanks to the identity of Proposition \ref{prop4} we obtain (\ref{dissipation}), which is positive because $\mathcal D^\alpha>0$, $\alpha = 1,2$ and $(x-y)\log(\frac{x}{y})$ is positive for all $x,y \geq 0$.

\end{proof}
\begin{remark}
We highlight that in the case in which $\mathbb{D}_{1,2}\textcolor{black}{(w)}=\mathbb{D}_{2,1}\textcolor{black}{(w)}=0$ and $\mathbb{D}\textcolor{black}{(w)}$ is isotropic, if we define an energy in the form
\begin{equation*}\label{energy}
\xi(w,t)=(U_p*f)(w,t)+\frac{\textrm{tr}(\mathbb{D}\textcolor{black}{(w)})}{2}\log(f\textcolor{black}{(w,t)})
\end{equation*}
which in our case corresponds to
\[
\mathcal{B}[f](w,t)=\nabla_w (U_p*f)(w,t),
\]
with $U_p=U_p(|w|)$ an interaction potential,
then we have that
\begin{equation*}\label{grad}
\nabla_w\xi(w,t)= \mathcal{B}[f](w,t)+\mathbb{D}\textcolor{black}{(w)}\nabla_w \log(f\textcolor{black}{(w,t)}).
\end{equation*}
Therefore, Eq. \eqref{FP1} may be written  in the form 
\begin{equation*}\label{grad.structure}
\partial_t f(w,t)=\nabla_{w}\cdot \left[ f(w,t)\nabla_w \xi(w,t)\right], \qquad w\in \Omega, 
\end{equation*}
and therefore in a gradient flow structure for which entropic averaged schemes may be used \cite{Pareschi_Zanella}.

\end{remark}

%%%%%%%%%%%%%%%%%%%%%%%%%%%%%%%%%%%

\section{Numerical tests}\label{sect:applications}
In this section we present some numerical examples of the class of Fokker-Planck \textcolor{black}{equations under study}  with nonconstant full diffusion matrix solved through structure-preserving schemes that have been introduced in the previous sections. 
We will approximate the long time behaviour of \eqref{FP1} with $d=2$, using the scheme defined by \eqref{eq:flux_th}-\eqref{Ltilde}-\eqref{delta}-\eqref{lambda}  with no-flux boundary conditions \eqref{no_flux.num}.
\textcolor{black}{In the following, we will show numerically how the high order approximation of the nonlinear weights \eqref{lambda} reflects in an improved accuracy of the large time behavior of  \eqref{FP1}. In particular, we consider open Newton-Cotes methods with $p=2,4,6$ points and we will also test a Gauss-Legendre quadrature.} For the Gaussian quadrature we considered 8 points in each numerical cell. In the sequel, we will adopt the notation $SP_k$, with  $k = 2,4,6,G$,  to denote the SP schemes with \eqref{lambda} that is evaluated with second, fourth, sixth order Newton-Cotes quadrature or Gaussian quadrature, respectively. We highlight how possible singularities at the boundaries are avoided using open quadrature rules. 

\subsection{Test 1. Validation}
In this subsection we consider a distribution function $f(w,t)$, $w \in[-1,1]\times [-1,1]$, whose evolution is given by \eqref{FP1} in which, given the diffusion matrix $\mathbb{D}$, we chose the drift operator in such a way that the flux vanishes. In particular, we consider a linear drift term in the form \eqref{w-U}
with a stationary state in the form 
\begin{equation}\label{ex.stat}
f_{\infty}(w)=C\exp\{-\phi(w)\},
\end{equation}
where $\phi(w)$ is a given function of the state variable, $C>0$ a normalization constant. Therefore the linear drift term will be in the form
\[
B(w):=-\nabla_w \cdot \mathbb{D}(w)-\mathbb{D}(w)\nabla_w \phi(w).
\]
This is the case in which we have entropy dissipation and convergence of order  $p$.
In particular, we shall consider $\mathbb D(w)$ a $2\times 2$ matrix of the form
\begin{equation}\label{eq:D_test1}
\mathbb{D}=
\left[\begin{matrix}
\dfrac{\sigma_1^2}{2}(1-w_x^2)^2 & \rho\dfrac{\sigma_1 \sigma_2}{4}(1-w_x^2)(1-w_y^2) 
\\ \rho\dfrac{\sigma_1\sigma_2}{4}(1-w_x^2)(1-w_y^2)  & \dfrac{\sigma_2^2}{2}(1-w_y^2)^2 
\end{matrix}\right],\qquad w_x,w_y\in[-1,1].
\end{equation}
As initial condition we consider 
\begin{equation}\label{initial}
f_0(w)=\beta\left[ \exp(-c(w_x+\mu)^2)\exp(-c(w_y+\mu)^2) +\exp(-c(w_x-\mu)^2)\exp(-c(w_y-\mu)^2\right],
\end{equation}
with $\mu = \dfrac{1}{2}$, $c=30$ and where $\beta>0$ is a normalization constant.
\begin{figure}[!htbp]
\centering
\includegraphics[scale=0.45]{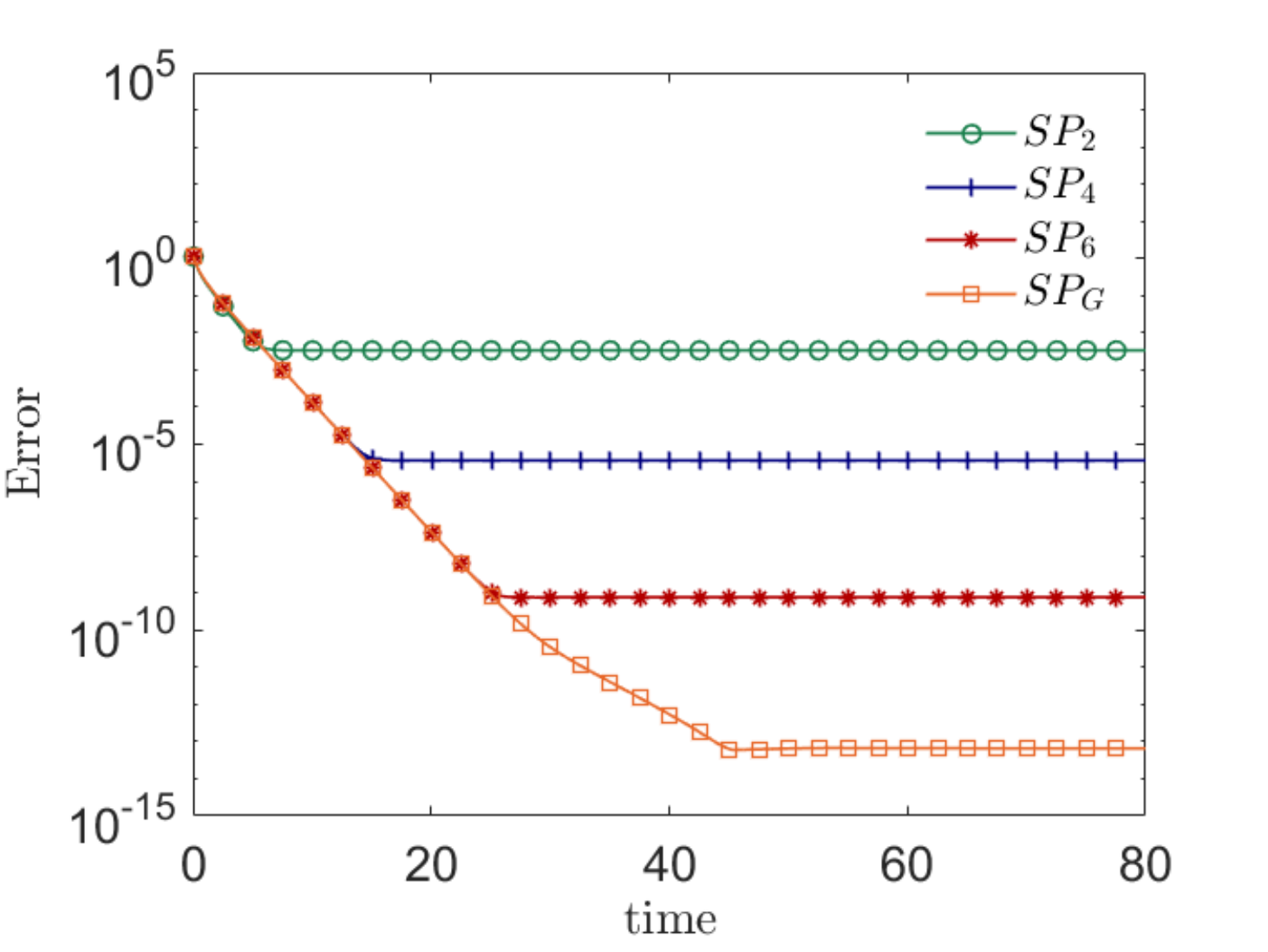}
\includegraphics[scale=0.45]{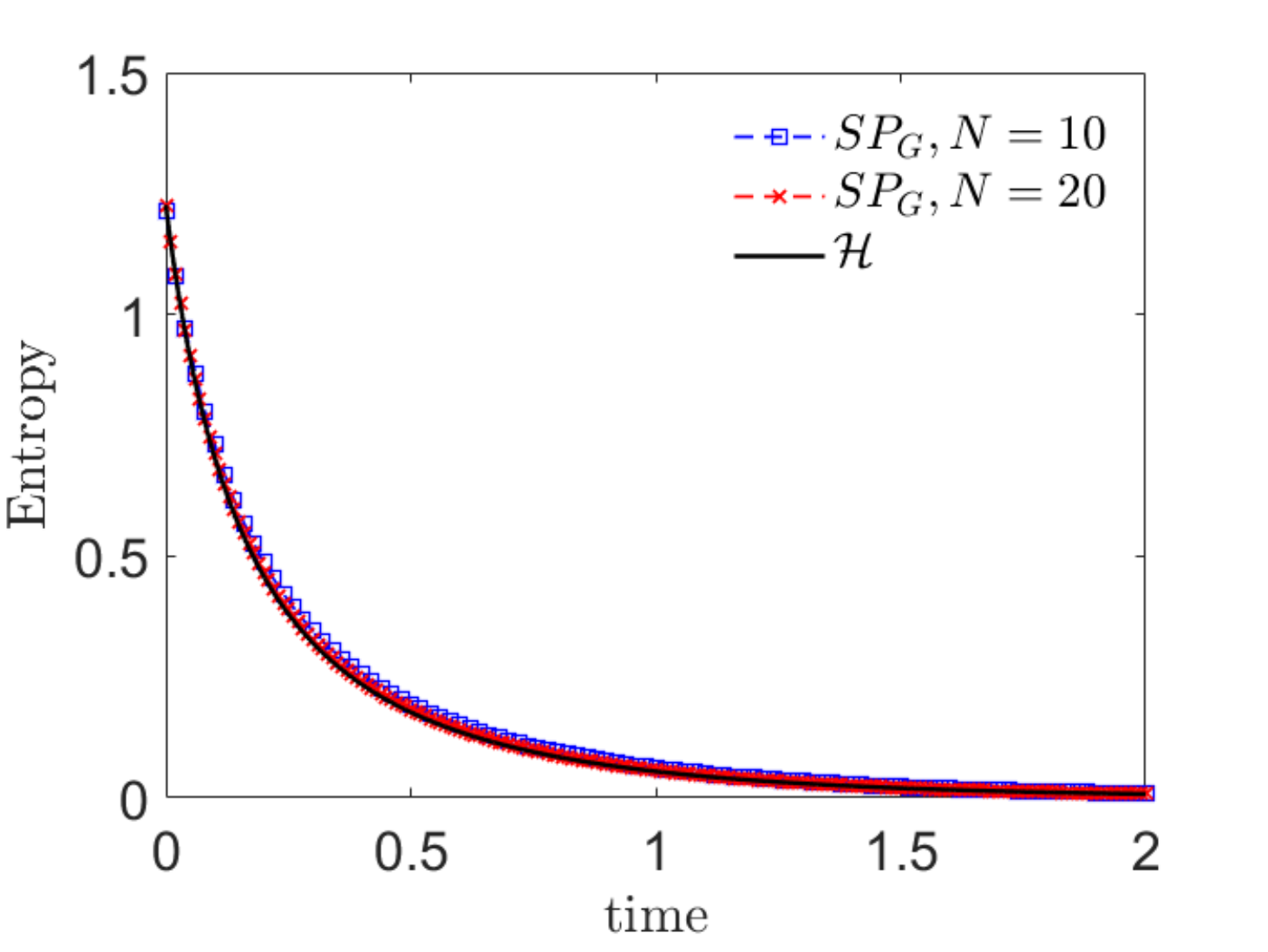}
\caption{\textbf{Test 1}. Left: evolution over the time interval $[0,80]$ of the relative \textcolor{black}{$L^1$ error  \eqref{err_L1}} computed with respect to the stationary solution \eqref{ex.stat} with $\phi(w)=-d\left(w_x^8+w_y^8\right)$, where $d=12.5$, for the $SP_k$ scheme with different quadrature methods. Initial distribution as in \eqref{initial} with $\sigma_1^2=\sigma_2^2=1$ and $\rho = 0.9$. We considered $\Delta t=\Delta w/(20\sigma_1^2)$, $\Delta w = 2/(N-1)$ and $N=81$. Right: dissipation of the numerical entropy for $SP_k$ scheme with Gaussian quadrature for two coarse grids with $N = 10$ and $N = 20$ points. }
\label{error1}
\end{figure}

\begin{table}
\centering
\begin{tabular}{l|| c c c c |  c c c c}
\hline\noalign\\[-10pt]
  & $SP_k$ & & & & $SP_k$ & & & \\ 
\hline
Time & 2 & 4 & 6 &G & 2 & 4 & 6 &G\\
\hline
1 & 1.9601& 1.6775&2.1106 &2.111 &1.9606 &1.8176  & 2.1015& 2.2103  \\
\hline
10&1.9662 & 3.9708& 7.4700 & 8.1449 & 1.9662&3.9708 & 7.4753& 8.1449 \\ 
\hline
20 &1.9662 & 3.9708 & 7.4768 & 8.1453 &1.9662 & 3.9708& 7.4760 & 8.1449  \\
\hline
\end{tabular}
\caption{\textbf{Test 1}. Estimation of the order of convergence  for $SP_k$ scheme with explicit Euler (left) and RK4 (right). Rates have been computed using $N = 21,41,81$ grid points in each component of the computational cell. We considered $\sigma_1^2=\sigma_2^2=1$, $\rho=0.1$, $\Delta t=\Delta w^2/(10\sigma_1^2\Delta w +10)$. }
\label{order2.esplicito}
\end{table}
In Figure \ref{error1} we compute the relative $L^1$ error of the numerical solution with respect to the exact stationary state $f^{\infty}$ given by \eqref{ex.stat}, i.e.
\textcolor{black}{\begin{equation}\label{err_L1}
e_r^N(t^n)=\dfrac{||f^n_{i,j}-f^{\infty}(w_{i,j})||_{L^1}}{||f^{\infty}(w_{i,j})||_{L^1}}
\end{equation}}
using $N = 81$ grid points for the $SP_k$ scheme with various quadrature rules. The different integration methods
capture the steady state with different accuracy. 
In particular low order quadrature rules
achieve their numerical steady state faster due to a saturation effect, whereas high order quadratures essentially reach machine precision in finite time. We considered in this plot semi-implicit time integration. In the same figure we illustrate how $SP_k$ scheme dissipates the relative entropy (\ref{entropy.dec}) in the case of two coarse grids with $N = 10$ and $N = 20$ points. 

In Table \ref{order2.esplicito} we estimate the order of convergence of the schemes for first order time integration and a fourth order Runge-Kutta integration that is computed as \textcolor{black}{$\log_2(e_r^N(T))$, with $N=81$ and $T$ is the final time of the numerical test}. The time step is chosen such that the CFL condition for the positivity of the scheme is satisfied, $\ie, \Delta t=\mathcal{O}(\Delta w^2)$. We may observe that in the transient regime the second order is maintained, whilst we reach higher orders for large times, expressing the order of the quadrature rules.
In Table \ref{order2} we estimate the order of convergence with first and
second order semi-implicit methods. We chose the time step $\Delta t = O(\Delta w) $ to meet the positivity bound derived in Proposition \ref{prop3}. Here again, we may observe that the scheme is second order accurate in the transient regime and describes the long time behaviour of the problem with the order employed for the evaluation of the nonlinear weights.

\begin{table}[!htbp]
\centering
\begin{tabular}{l|| c c c c |  c c c c}
\hline\noalign\\[-10pt]
$\rho = 0.1$ & $SP_k$ & & & & $SP_k$ & & & \\ 
\hline
Time & 2 & 4 & 6 &G & 2 & 4 & 6 &G\\
\hline
1&1.9625&1.4962 &1.6460& 1.6461& 1.9629&1.7472 &1.8889 & 1.8891  \\
\hline
10& 1.9662 &3.9708&7.3407&7.9144& 1.9662& 3.9708  &7.4765 & 7.8903  \\
\hline
20& 1.9662 & 3.9708& 7.4769& 7.9144& 1.9662& 3.9708 & 7.4772& 8.1457 \\ 
\hline \\
$\rho = 0.9$ & $SP_k$ & & & & $SP_k$ & & & \\ 
\hline
Time & 2 & 4 & 6 &G & 2 & 4 & 6 &G\\
\hline
1& 1.8570& 1.9049 & 1.9100 & 1.9100 &1.8878& 1.9559 & 1.9622 & 1.9622 \\
\hline
10& 1.9621&3.9678  &2.1457 & 2.1554 &1.9621&  4.0880& 2.4631  & 7.4904 \\
\hline
20& 1.9621& 3.9800 & 6.0613& 7.2470&1.9621 & 3.9800  & 6.0649 &  7.2697\\ 
\hline
50& 1.9621 & 3.9800 & 6.2146 &7.8973&1.9621 & 3.9800 & 6.2144 &  7.8964 \\
\hline

\end{tabular}
\caption{\textbf{Test 1}. Estimation of the order of convergence  for $SP_k$ scheme with first (left) and second order (right) semi-implicit methods. Rates have been computed using $N = 21,41,81$ grid points, $\sigma_1^2=\sigma_2^2=1$, $\Delta t=\Delta w/(20\sigma_1^2)$, and two correlation coefficients $\rho = 0.1$ (top) and $\rho = 0.9$ (bottom).}
\label{order2}
\end{table}

\subsection{Test 2. Alignment dynamics in bounded domains}
In this test we provide numerical evidence of the failure of the scheme on models with non-vanishing flux at equilibrium. Let us consider the evolution of a distribution function as in \eqref{FP1} with $w\in[-1,1]\times[-1,1]$, anisotropic diffusion introduced in \eqref{eq:D_test1}, and
\begin{equation}\label{eq:B_test2}
\mathcal{B}[f](w,t)=\int_{[-1,1]\times [-1,1]}P(w,w_*)(w-w_*)f(w_*,t)dw_*,
\end{equation}
with initial distribution of the form \eqref{initial}. We note that in this case we have no guarantee that the flux vanishes for large times. 

 \textcolor{black}{First of all we consider \eqref{eq:B_test2} with $P\equiv  1$. This corresponds to $\mathcal{B}[f](w,t)=B(w)=(w-U)$, $U=\int_\Omega f(w_*,t) w_*\, dw_*$ that is constant \textcolor{black}{since the mean }of $f$ is conserved. \textcolor{black}{It is worth to notice that we are in the case} in which a stationary state making the flux vanish may exist and we have decay of entropy.} \textcolor{black}{Since the stationary state of the problem is not known analytically, we computed the relative $L^1$ error for successive approximations. We denote with $f^{N_s}$ the approximation of $f$ done using a grid with $N_s$ points and we compute the error by considering as reference solution the one of the successive refinement of the computational grid, i.e.
 \begin{equation}\label{err_rel_succ}
 e_s(t^n)=\dfrac{||f^{N_s,n}_{i,j}-f^{N_{s-1},n}_{i,j}||_{L^1}}{||f^{N_{s+1},n}_{i,j}-f^{N_s,n}_{i,j}||_{L^1}}.
\end{equation}  
In detail, we chose $N_1= 21$, $N_2 = 41$, $N_3 = 81$ grid points. The order of convergence is then computed as $\log_2(e_s(t^n))$.} In Table \ref{order4.rk4} we estimate the order of convergence of the $SP_k$ scheme with explicit time integration methods. In particular, we present the case of first order forward Euler method and fourth order Runge-Kutta with a suitable time step to guarantee positivity of the scheme, i.e. $\Delta t = O(\Delta w^2)$. In Table \ref{order4} we estimate the order of convergence of the method in the case of semi-implicit time integration taking into account first and second order semi-implicit methods with $\Delta t = O(\Delta w)$. \textcolor{black}{We may observe that in both cases, the proposed scheme is not capable to approximate the large time solution of the problem with high order. Indeed we have no theoretical guarantee that the flux vanishes at the equilibrium that is the assumption under which the scheme has been derived.} \textcolor{black}{The time evolution of the approximated solution are represented in Figure \ref{swarming.bounded}.}
 
\begin{table}
\centering
\begin{tabular}{l|| c c c c |  c c c c}
\hline\noalign\\[-10pt]
  & $SP_k$ & & & & $SP_k$ & & & \\ 
\hline
Time & 2 & 4 & 6 &G & 2 & 4 & 6 &G \\
\hline
1 &2.0830 &2.1102 &2.3204 &2.4229  & 2.1320& 2.3606 &2.3602 & 2.3602 \\
\hline
10 &2.0914 &2.2000 & 2.3614& 2.5143 & 2.4199& 2.8006& 2.8195& 2.8199 \\ 
\hline
20&2.0914 &3.7579 &4.0746 &3.8000  &2.8741 &3.7503 & 3.9163 & 3.8875  \\
\hline
\end{tabular}
\caption{\textbf{Test 2}. Estimation of the order of convergence  for $SP_k$ scheme with explicit Euler (left) and RK4 (right). Rates have been computed using $N = 21,41,81$ grid points in each component of the computational cell. We considered $P\equiv 1$, $\sigma_1^2=\sigma_2^2=1$, $\rho=0.1$, $\Delta t=\Delta w^2/(10\sigma_1^2\Delta w +10)$. }
\label{order4.rk4}
\end{table}

\begin{table}
\centering
\begin{tabular}{l|| c c c c|c c c c}
\hline\noalign\\[-10pt]
  & $SP_k$ & & & & $SP_k$ & & &  \\ 
\hline
Time & 2 & 4 & 6 &G & 2 & 4 & 6 &G\\
\hline
1  & 1.9585& 2.0242 &2.2398 & 2.2615 & 1.9612& 2.1190 & 2.2398 & 2.2732 \\
\hline
10 & 2.0694& 3.9977& 3.6949& 3.6477 & 2.0685 & 3.9643 & 3.6601&  3.6140\\ 
\hline
20& 2.0695 & 3.9982 & 3.6957 & 3.6486 & 2.0686 & 3.9643 & 3.6608 &  3.6140 \\
\hline
\end{tabular}
\caption{\textbf{Test 2}. Estimation of the order of convergence  for $SP_k$ scheme with first (left) and second order (right) semi-implicit integration. Rates have been computed using $N = 21,41,81$, $P\equiv 1$, $\sigma_1^2=\sigma_2^2=1$, $\rho=0.1$, $\Delta t=\Delta w/(20\sigma_1^2)$. }
\label{order4}
\end{table}

\begin{figure}
\centering
\subfigure[$t = 0.2$]{\includegraphics[scale=0.32]{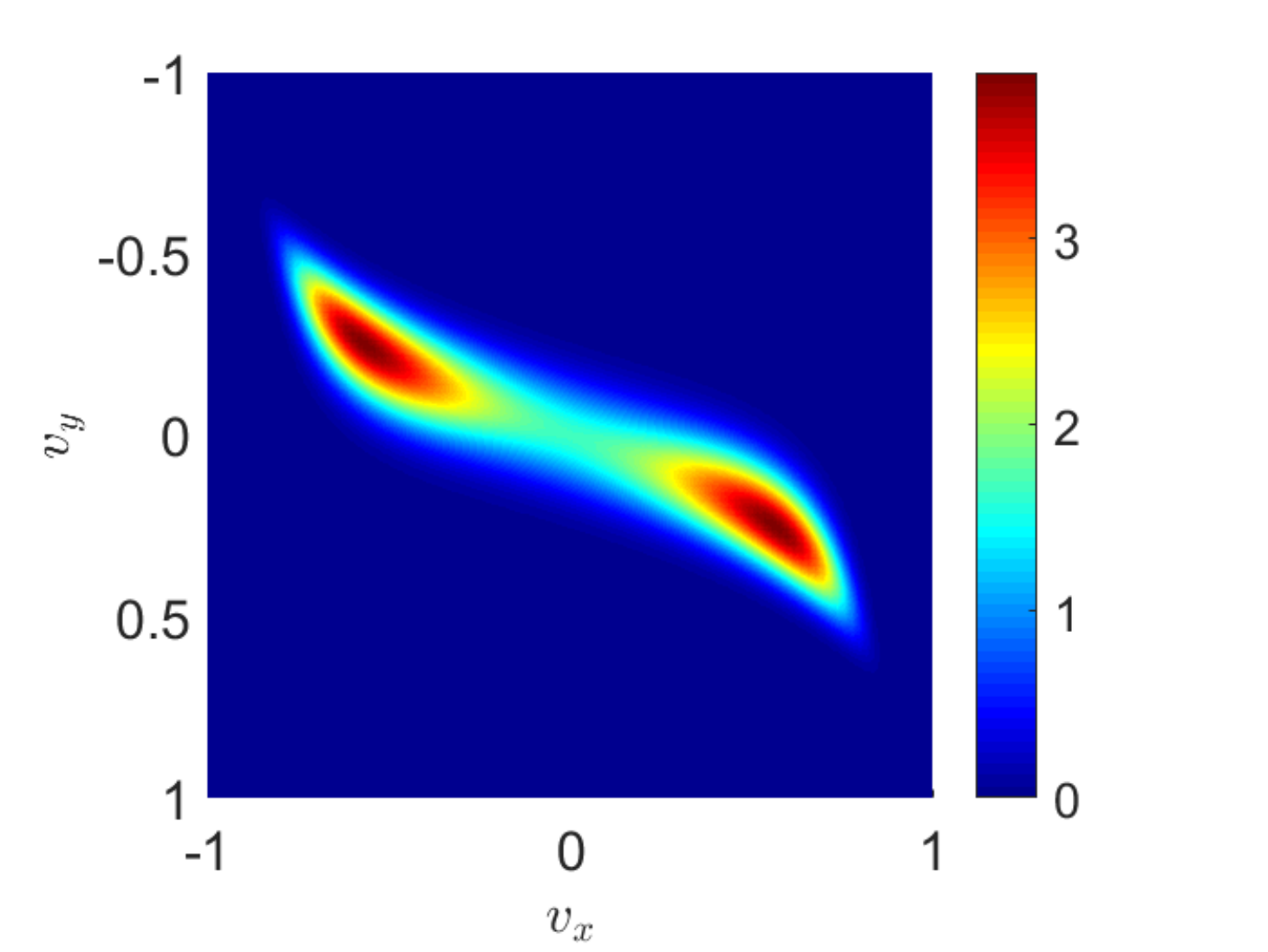}}
\subfigure[$t = 0.4$]{\includegraphics[scale=0.32]{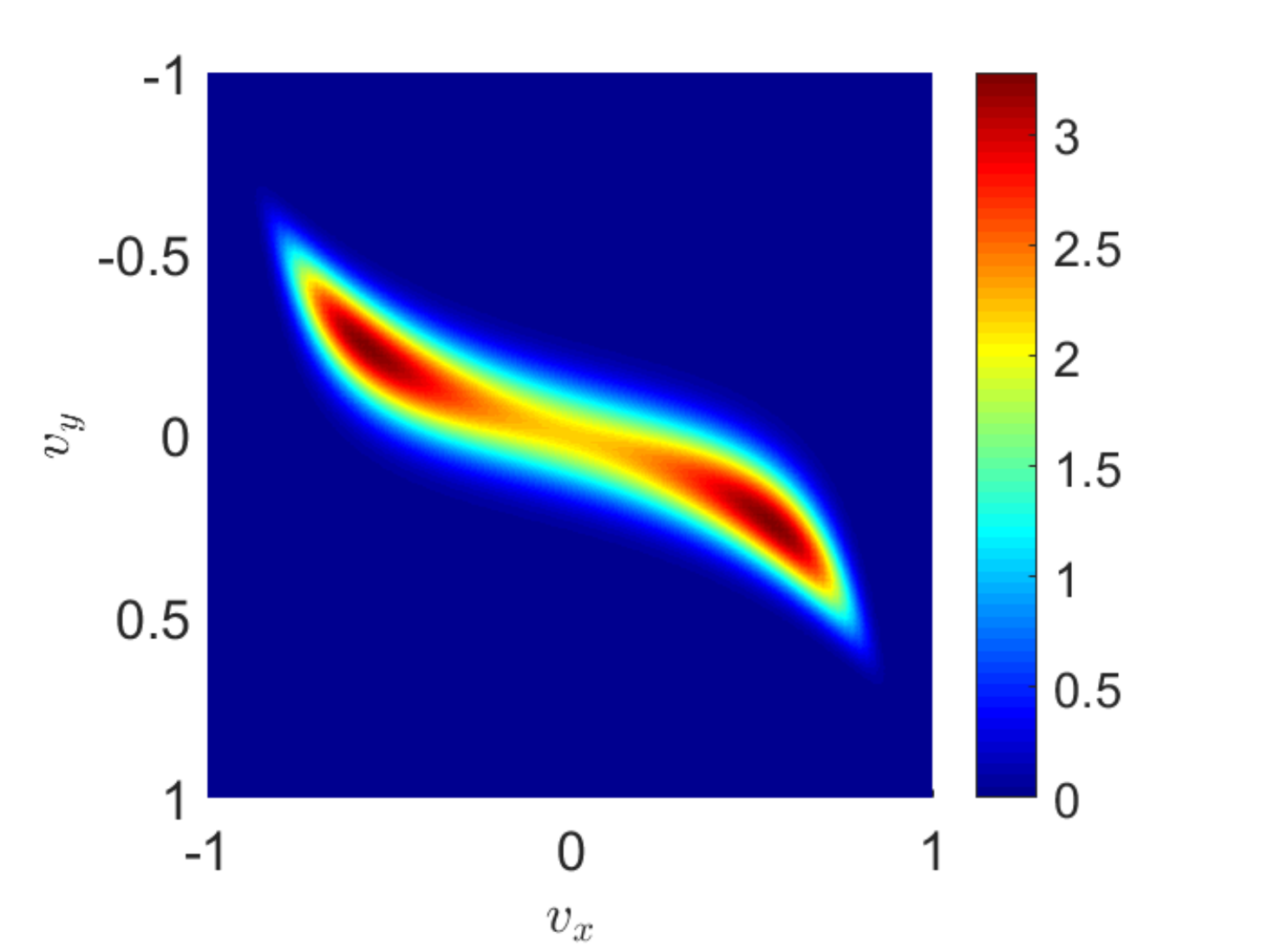}} 
\subfigure[$t = 1$]{\includegraphics[scale=0.32]{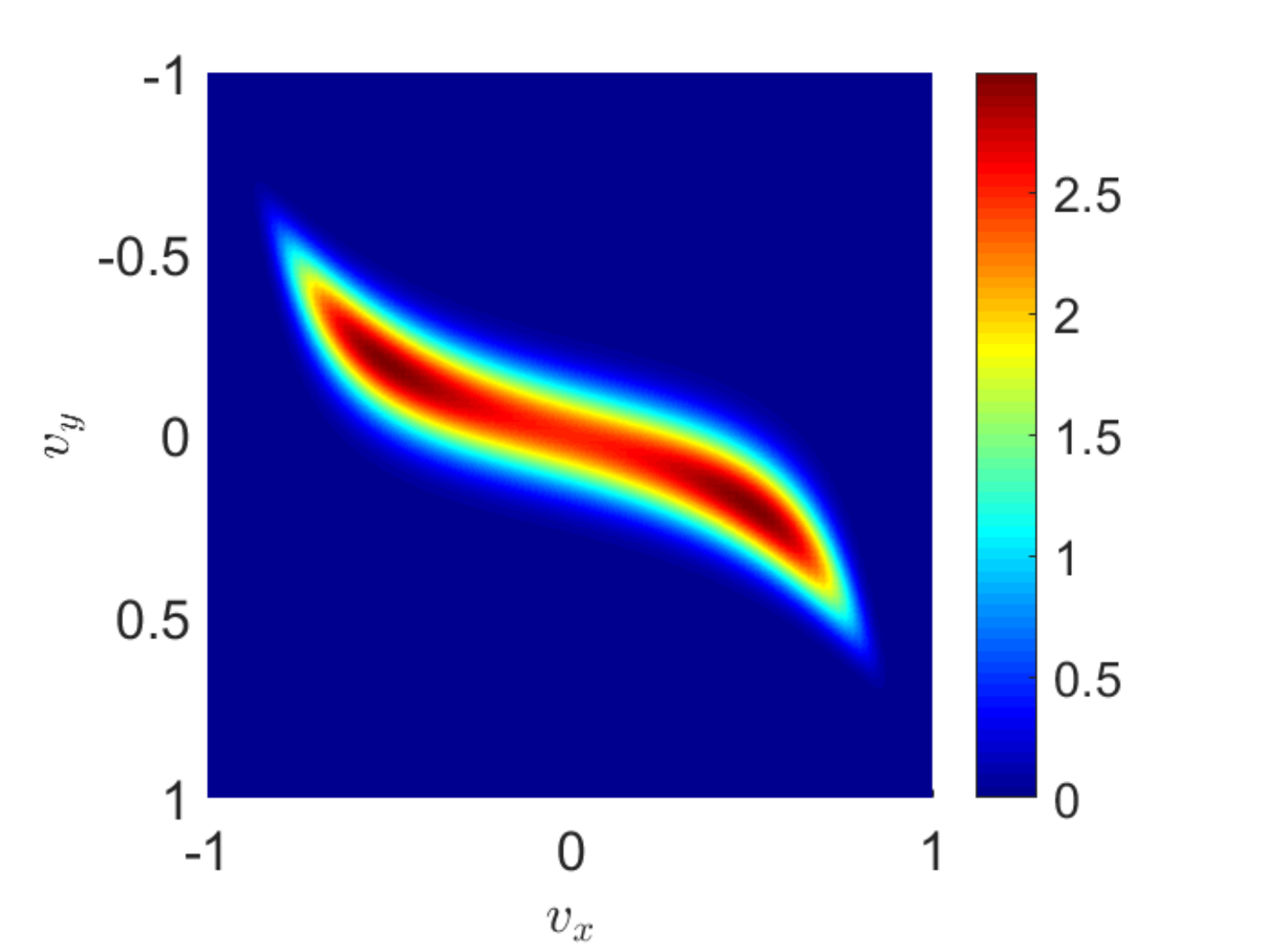}}\\
\subfigure[$t = 0.03$]{\includegraphics[scale=0.32]{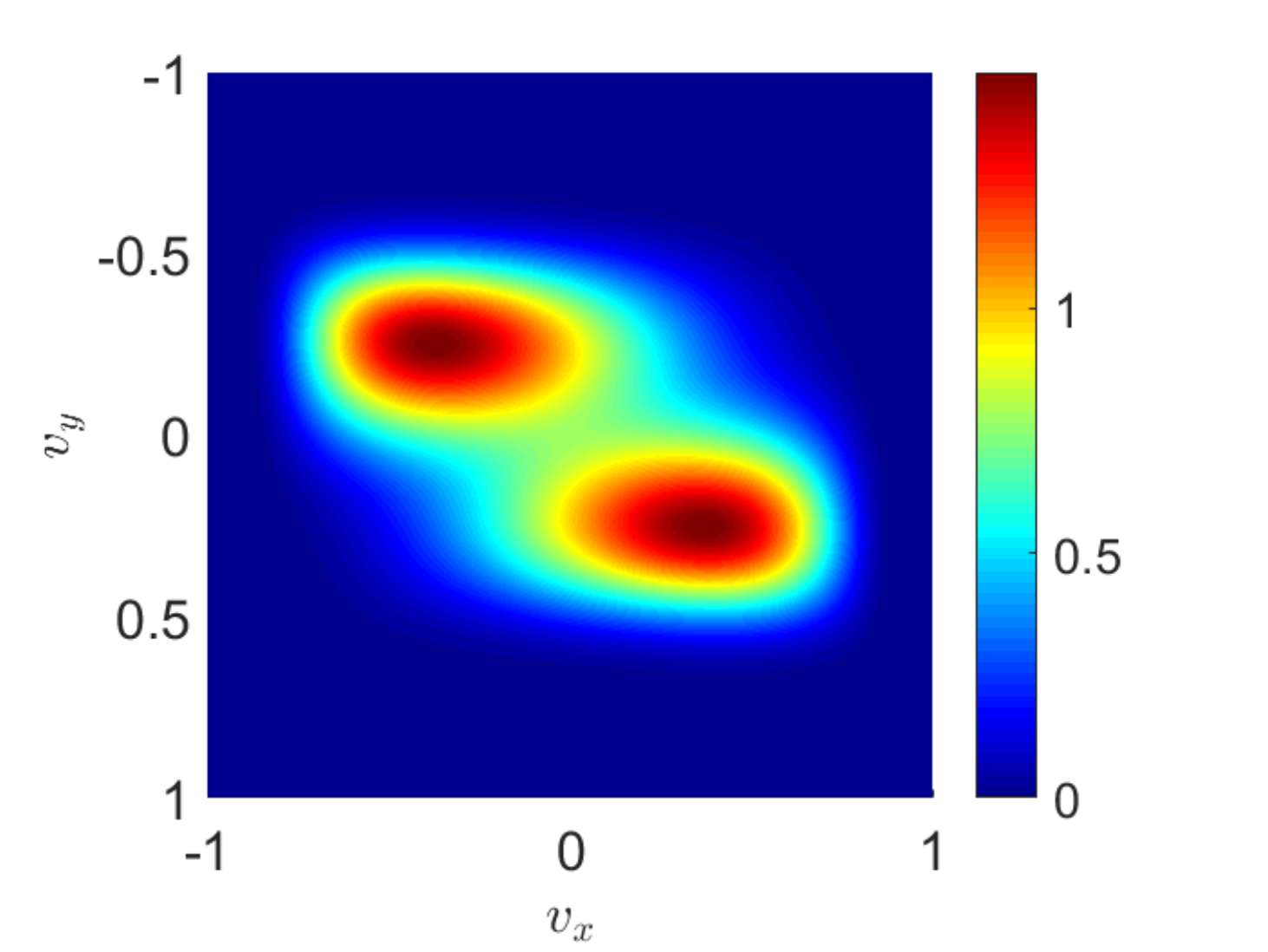}}
\subfigure[$t = 0.05$]{\includegraphics[scale=0.32]{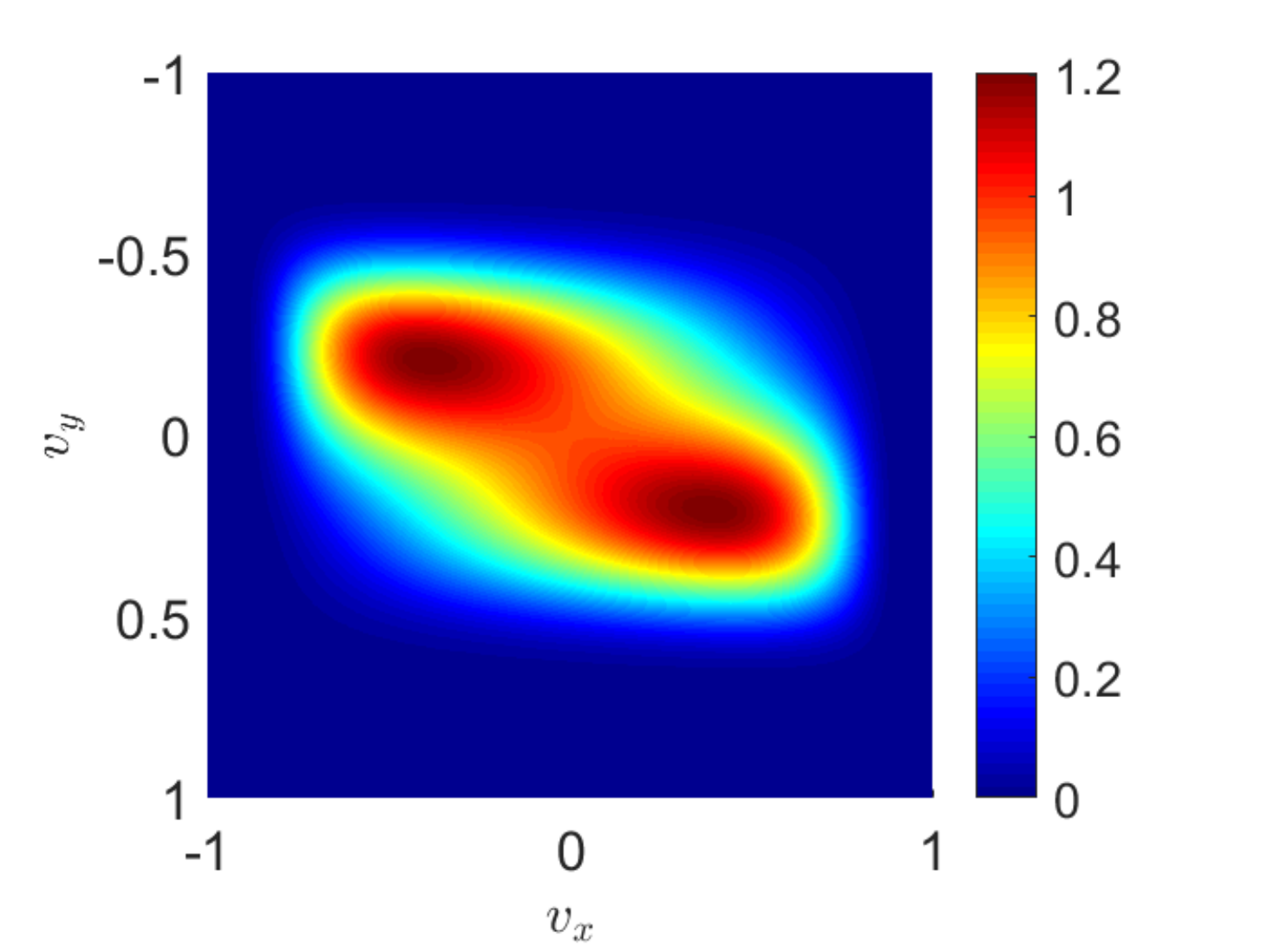}} 
\subfigure[$t = 0.2$]{\includegraphics[scale=0.32]{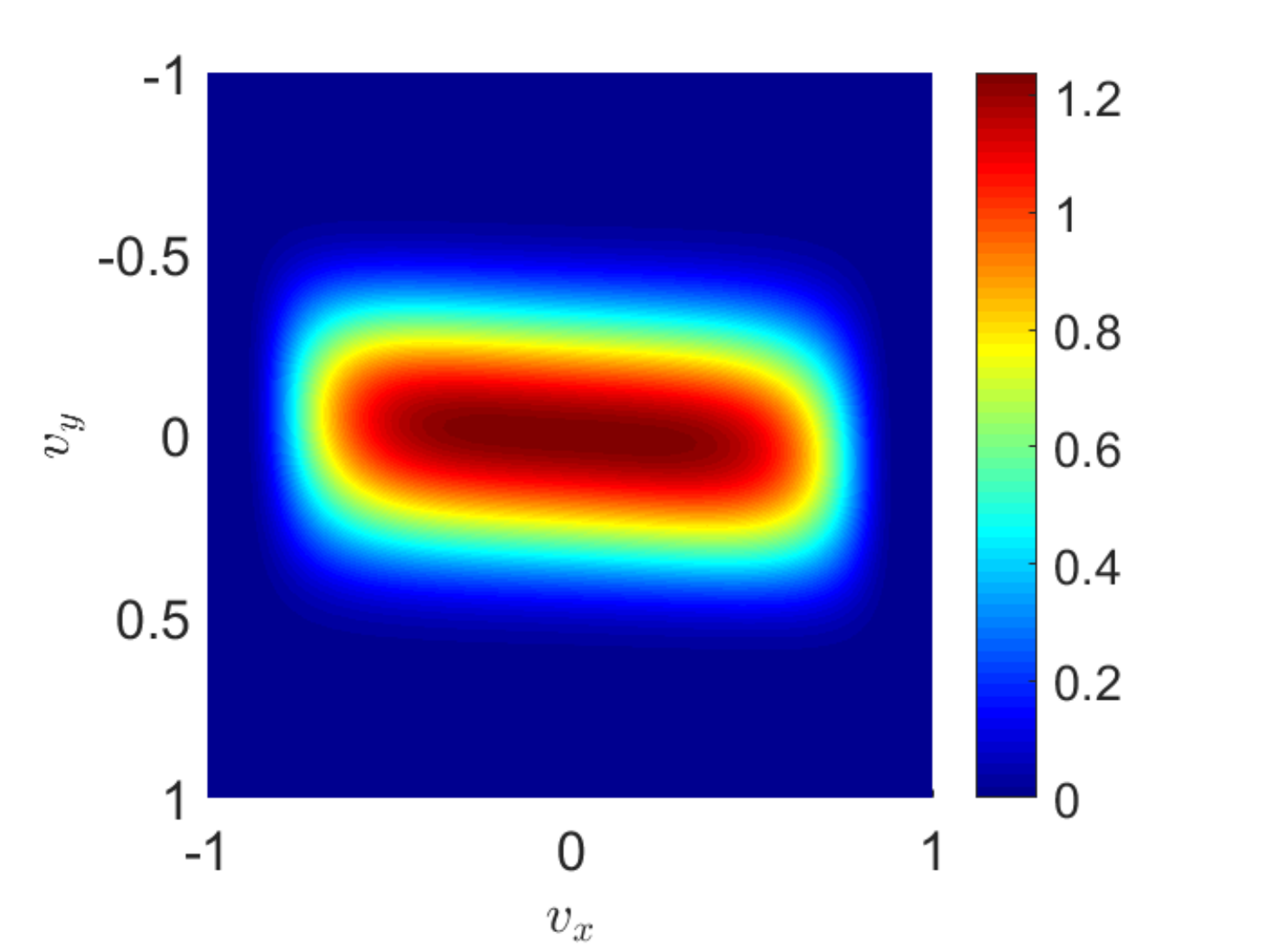}}\\
\caption{\textbf{Test 2}. Evolution of the numerical solution of the nonlinear Fokker-Planck equation with drift \eqref{eq:B_test2}, $P\equiv 1$, and anisotropic diffusion matrix \eqref{eq:D_test1} with $\sigma_1^2 = 0.1$, $\sigma_2^2 = 0.5$ and correlation coefficient $\rho = 0.9$ (top row) and $\rho = 0.1$ (bottom row).  The numerical solution has been computed with $N = 101$ grid points in both components and semi-implicit time integration with $\Delta t = \Delta w/(20 \max\{\sigma_1^2,\sigma_2^2\})$.}
\label{swarming.bounded}
\end{figure}

%%%%%%
\section*{Conclusion}
We studied the construction of structure preserving methods for a class of \textcolor{black}{two-dimensional} Fokker-Planck equations with full nonconstant diffusion matrix. We have derived the schemes for stationary states that make the flux vanish. We have proved that mass conservation and positivity of the solution both with explicit and semi-implicit time integration hold even for problems with a non-vanishing flux at the steady state. Furthermore, the methods here developed are positivity preserving without any restriction on the discretization of the state variable both in the case of explicit and of semi-implicit time integration methods, the latter in particular lead to more mild restrictions on the time step that are very useful in the high-dimensional case. \textcolor{black}{On the other hand, the evolution scheme is in general equilibrium preserving for Fokker-Planck-type equations with nonconstant diffusion matrix if the drift is such that the flux function vanishes at the steady state and if it does not depend on the solution. } Under these assumptions we also showed entropy decay of the problem and that that the introduced scheme dissipates the numerical entropy. We also presented numerical evidence of the theoretical findings. \textcolor{black}{Extensions of the proposed scheme to dimensions higher than two are currently under study and will be presented elsewhere.} 

\section*{Acknowledgements}
This research was partially supported by the Italian Ministry of Education, University and Research (MIUR) through the ``Dipartimenti di Eccellenza'' Programme (2018-2022) -- Department of Mathematical Sciences ``G. L. Lagrange'', Politecnico di Torino (CUP: E11G18000350001).

Both authors are members of GNFM (Gruppo Nazionale per la Fisica Matematica) of INdAM (Istituto Nazionale di Alta Matematica), Italy.

NL would like to thank Compagnia San Paolo for financing her PhD scholarship.

\end{document}